\documentclass[utf8,12pt,nothms,a4paper,english]{aclart}

 \usepackage{hyperref}

 \usepackage{lmodern}
 \usepackage{color}

\numberwithin{equation}{subsection}
\theoremstyle{plain}

\newtheorem{theo}[subsection]{\theoname}
\newtheorem{lemm}[subsection]{\lemmname}
\newtheorem{prop}[subsection]{\propname}
\newtheorem{coro}[subsection]{\coroname}
\theoremstyle{remark}
\newtheorem{rema}[subsection]{\remaname}

\makeatletter
\let\c@paragraph\c@equation

\let\cl@paragraph\cl@equation
\makeatother
\def\sozat{\,;\,}

\RequirePackage{tikz,tikz-cd}
\usetikzlibrary{matrix,arrows,decorations.pathmorphing}

\theoremstyle{plain}
\newtheorem*{theo*}{\theoname}
\theoremstyle{remark}

\def\setminus{\mathchoice
    {\mathbin{\vrule height .72ex width 1.61ex depth -.38ex}}
    {\mathbin{\vrule height .72ex width 1.61ex depth -.38ex}}
    {\mathbin{\vrule height .50ex width 0.85ex depth -.28ex}}
    {\mathbin{\vrule height .20ex width 0.570ex depth -.24ex}}
}
\def\abs#1{\left|#1\right|}

\def\C{\mathbf C}
\def\R{\mathbf R}
\def\N{\mathbf N}
\def\Z{\mathbf Z}
\def\Q{\mathbf Q}
\def\P{\mathbf P}

\let\bar\overline
\let\phi\varphi \let\eps\varepsilon
\let\emptyset\varnothing
 
\def\lbra{\mathopen{[\![}} \def\rbra{\mathclose{]\!]}}

\def\Aut{\operatorname{Aut}}
\def\Card{\operatorname{Card}}
\def\ord{\operatorname{ord}}

\def\PGL{\operatorname{PGL}}
\def\GL{\operatorname{GL}}
\def\SL{\operatorname{SL}}
\def\End{\operatorname{End}}
\def\id{\operatorname{id}}
\def\an{{\mathrm{an}}}

\def\sing{{\mathrm{sing}}}
\def\MAT{{\mathrm{M}}}

\title {A non-archimedean Ax-Lindemann theorem}
\author{Antoine Chambert-Loir}
\address{%
Univ. Paris Diderot, Sorbonne Paris Cité, Institut de Mathématiques de Jussieu-Paris Rive Gauche, UMR 7586, F-75013, Paris, France}
\email{Antoine.Chambert-Loir@math.univ-paris-diderot.fr}
\author{Fran\c cois Loeser}
\address{Sorbonne Universit\'es, UPMC Univ Paris 06, UMR 7586 CNRS, Institut Math\'ematique de Jussieu-Paris Rive Gauche, F-75005, Paris, France}
\email{Francois.Loeser@upmc.fr}
\dedicatory{\`A Daniel Bertrand, en témoignage d'amitié}
\begin{document}


\maketitle

\section{Introduction}
\subsection{}
The classical Lindemann–Weierstrass theorem states that if algebraic numbers
$\alpha_1,\dots,\alpha_n$  are $\Q$-linearly independent, then their exponentials
$\exp(\alpha_1),\dots,\exp(\alpha_n)$
are algebraically independent over $\Q$. More generally, if $\alpha_1,\dots,\alpha_n$ are any $\Q$-linearly independent complex numbers,
no longer assumed to be algebraic, 
Schanuel's conjecture 
predicts that  the field $\Q(\alpha_1,\dots,\alpha_n, \exp(\alpha_1),\dots,\exp(\alpha_n))$ has transcendence degree at least $n$ over $\Q$.
In \cite{ax:1971}, Ax established power series and differential field versions of Schanuel's conjecture.
In particular the part of Ax's results corresponding to the Lindemann–Weierstrass theorem
can be recasted into geometrical terms as follows:

\begin{theo}[Exponential Ax-Lindemann]
Let $\exp: \C^n \to (\C^{\times})^n$ be the morphism
$(z_1,\dots, z_n)\mapsto(\exp(z_1),\dots,\exp(z_n))$.
Let $V$ be an irreducible algebraic subvariety of 
$(\C^{\times})^n$ and let $W$ be an irreducible component of a maximal algebraic
subvariety of~$\exp^{-1}(V)$. Then $W$~is geodesic, that is, $W$ is defined by a finite family of
equations of the form $\sum_{i=1}^n a_i z_i = b$ 
with $a_1,\dots,a_n \in \Q$ and $b \in \C$.
\end{theo}

In the breakthrough paper \cite{pila:2011}, Pila succeeded in providing an unconditional proof of the André-Oort conjecture for
products of modular curves. One of  his  main ingredients was to prove a hyperbolic  version of the above Ax-Lindemann theorem, which we now state in a simplified version.

Let ${\mathbf{h}}$ denote the complex upper
half-plane and $j\colon  \mathbf{h} \to \C$
the elliptic modular function. 
By an algebraic subvariety of $\mathbf h^n$,
we shall mean the trace in ${\mathbf{h}}^n$ of
an algebraic subvariety of $\C^n$.
An algebraic subvariety of ${\mathbf{h}}^n$
is said to be geodesic if it can be
defined by equations of the form
$z_i = c_i$ and $z_k = g_{k\ell} z_{\ell}$, with $c_i \in \C$ and
$g_{k\ell} \in \GL(2,\Q)^+$.

\begin{theo}[Hyperbolic Ax-Lindemann]
Let $j\colon \mathbf{h}^n \to \C^n$
be the morphism
$(z_1,\dots, z_n)\mapsto(j(z_1),\dots,j(z_n))$.
Let $V$ be an irreducible algebraic subvariety of~$\C^n$ 
and let $W$ be an irreducible component of a maximal algebraic
subvariety of~$j^{-1}(V)$. Then $W$~is geodesic.
\end{theo}

Pila's method to prove this Ax-Lindemann theorem is quite different from the differential approach of Ax.
It follows a strategy initiated by Pila and Zannier in their new proof 
 of the Manin-Mumford conjecture for abelian
varieties \cite{pila-zannier:2008}; that approach makes crucial
use of  the bound on the number of rational points of  bounded height in the transcendental part of sets definable in an o-minimal structure
obtained by Pila and Wilkie in \cite{pila-wilkie2006}.
Recently, still using the Pila and Zannier strategy, Klingler, Ullmo and Yafaev have succeeded in proving a very general form of the
hyperbolic Ax-Lindemann theorem valid for  any arithmetic variety (\cite{klingler-ullmo-yafaev2015}, see also \cite{ullmo-yafaev2014} for the compact case).

\subsection{}In the recent paper \cite{cluckers-comte-loeser2015}, Cluckers, Comte and Loeser established a non-archimedean analogue of the Pila-Wilkie  theorem of \cite{pila-wilkie2006} 
in its block version of \cite{pila2009}.
The purpose of this paper is to use this result to prove a version 
of Ax-Lindemann for products of algebraic curves
admitting a non-archimedean uniformization and whose corresponding
Schottky group is ``arithmetic'' and has rank at least~$2$ (theorem~\ref{mt}).
In particular, this theorem applies  for products of 
Shimura curves admitting a $p$-adic uniformization à la
Cherednik-Drinfeld (see section~\ref{sec.cherednik}).

The basic strategy we use is strongly inspired by that of Pila \cite{pila:2011} (see also \cite{pila2015}), though some new ideas are required in order to adapt it to the
non-archimedean setting. Similarly as  in Pila's approach one starts by working on some neighborhood of the boundary of our space (which, instead
of a product of Poincaré upper half-planes, is a product of open subsets
of the Berkovich projective line). Analytic continuation and monodromy arguments are replaced
by more algebraic ones and explicit matrix computations by   group theory considerations.
We also take advantage of the fact that Schottky groups are free and of the geometric description of their fundamental domains.
Compared with Pila's proof,
where parabolic elements are used in a crucial way,
one main difficulty of the nonarchimedean situation
lies  in the fact that
all nontrivial elements of a Schottky groups are hyperbolic.

To conclude, let us note that there are cases where
$p$-adic analogues of theorems in transcendental number theory 
seem to require other methods than those used to prove their 
complex counterparts.
For instance, it is still an open problem to prove a
$p$-adic analogue, for values of the $p$-adic exponential function, 
of the classical Lindemann-Weierstrass theorem.

Since his first works (see, for example, \cite{bertrand76}),
Daniel Bertrand has shown deep insight into $p$-adic
transcendental number theory, 
and disseminated his vision within the mathematical community.
We are pleased to dedicate this paper to him. 

\subsection*{Acknowledgements}
The research leading to this paper was initiated during the 2014 MSRI program ``Model Theory, Arithmetic Geometry and Number Theory''.
We would like to thank the MSRI for its congenial atmosphere and hospitality.
This research was partially supported by ANR-13-BS01-0006 (Valcomo) and by ANR-15-CE40-0008 (Défigéo).
The second author
was partially supported  by the European Research Council 
under the European Community's Seventh Framework Programme (FP7/2007-2013)/ERC Grant Agreement nr.~246903 NMNAG and the Institut Universitaire de France.

During the early stage of this project, the second author 
benefited from stimulating discussions with Barry Mazur at MSRI, to whom we express our heartfelt thanks.
We are also grateful to Daniel Bertrand for continuous support 
and encouragement.

The comments of Yves André, Jean-François Boutot,
Zoé Chatzidakis, Antoine Ducros, Florent Martin and Jonathan Pila,
helped us to improve this paper, as
well as the numerous remarks of a diligent referee; 
we thank them warmly.

\section{Statement of the theorem}

\subsection{Non-archimedean analytic  spaces}
Given  a complete non-archimedean valued field~$F$, we shall consider in this
paper $F$-analytic spaces in the sense of Berkovich~\cite{berkovich1990,berkovich1993}.
However, the statements, and essentially the proofs, can be carried
on \emph{mutatis mutandis} in the rigid analytic setting.
In this context, there is a notion of irreducible component
(see~\cite{ducros2009}, or~\cite{conrad1999} for the rigid analytic version).

If $V$ is an algebraic variety over~$F$, 
we denote by $V^\an$ the corresponding $F$-analytic space.
There is a canonical topological embedding of~$V(F)$ in~$V^\an$,
and its image is closed if $F$ is locally compact.

If $F'$ is a complete non-archimedean extension of~$F$,
we denote by~$X_{F'}$ the $F'$-analytic space
deduced from an $F$-analytic space~$X$ by base change to~$F'$.

\subsection{Schottky groups}
Let $p$ be a prime number; we denote by~$\C_p$ the completion
of an algebraic closure of~$\Q_p$ and let $F$ be a finite extension of~$\Q_p$
contained in~$\C_p$.
The group $\PGL(2,F)$ acts by homographies on
the $F$-analytic projective line~$\P_1^\an$.
In the next paragraphs,
we recall from~\cite{gerritzen-vanderput1980} a few definitions
concerning Schottky groups  in~$\PGL(2,F)$, their limit sets
and the associated uniformizations of algebraic curves.

One says that a discrete subgroup~$\Gamma$ of $\PGL(2,F)$
is a \emph{Schottky group} if it is
finitely generated, and if no element ($\neq \id$) has finite order
\cite[I, (1.6)]{gerritzen-vanderput1980}.
If $\Gamma$ is a Schottky group, then $\Gamma$ is free;
moreover, any discrete finitely generated subgroup of $\PGL(2,F)$
possesses a normal subgroup of finite index which is a Schottky group,
\cite[I, (3.1)]{gerritzen-vanderput1980}.

We say that $\Gamma$ is \emph{arithmetic} if its elements can be
represented by matrices whose coefficients lie in a number field.
In this case, it follows from the hypothesis that $\Gamma$ is finitely generated
that there exists a number field~$K\subset F$ such that $\Gamma\subset\PGL(2,K)$.

\subsection{Limit sets}
Let $\Gamma$ be a Schottky subgroup of $\PGL(2,F)$.
Its \emph{limit set} is the set~$\mathscr L_\Gamma$
of all points in~$\P_1(\C_p)$
of the form $\lim_n (\gamma_n\cdot x)$,
where $(\gamma_n)$ is a sequence of distinct elements
of~$\Gamma$ and $x\in\P_1(\C_p)$,
\cite[I, (1.3)]{gerritzen-vanderput1980}.
By
\cite[I, (1.6)]{gerritzen-vanderput1980},
the limit set $\mathscr L_\Gamma$ is a compact subset of~$\P_1(F)$.
If the rank of~$\Gamma$ is at least~$2$, then $\mathscr L_\Gamma$
is a perfect (that is, closed and without isolated points)
subset of~$\P_1(F)$,
see \cite[I, (1.6.3) and (1.7.2)]{gerritzen-vanderput1980}.

Let $\Omega_\Gamma=(\P_1)^\an\setminus\mathscr L_\Gamma$;
it is a $\Gamma$-invariant open set of~$\P_1^\an$.
By lemma~\ref{lemm.curve-sparse} below, it  is geometrically irreducible.

\subsection{Quotients}
Let us  assume that $\Gamma$ is a Schottky group and let $g$ be its rank.
From the explicit description of the action of the group~$\Gamma$
given by \cite[I, \S4]{gerritzen-vanderput1980} and recalled
in~\S\ref{ss.fund-domain} below, see also~\cite[p.~86]{berkovich1990},
it follows that the group~$\Gamma$ acts freely 
on~$\Omega_\Gamma$, and the quotient space 
$\Omega_\Gamma/\Gamma$
admits a unique structure of
an $F$-analytic space
such
that the projection $p_\Gamma\colon \Omega_\Gamma\to \Omega_\Gamma/\Gamma$
is both a topological covering and a local isomorphism.
Moreover,
$\Omega_\Gamma/\Gamma$ is the $F$-analytic space associated with a smooth, 
geometrically connected, projective $F$-curve~$X_\Gamma$ of genus~$g$
(\cite[III, (2.2)]{gerritzen-vanderput1980}
and \cite[theorem~4.4.1, p.~86]{berkovich1990}),
canonically determined by the \textsc{gaga} theorem in this context,
\cite[theorem~3.4.12, p.~68]{berkovich1990}.

 
\subsection{}
Let us now consider a finite family $(\Gamma_i)_{1\leq i\leq n}$
of Schottky subgroups of $\PGL(2,F)$ of rank~$\geq 2$.
Let us set $\Omega=\prod_{i=1}^n \Omega_{\Gamma_i}$ and
$X=\prod_{i=1}^n X_{\Gamma_i}$, 
and let $p\colon \Omega \to X^\an$ be the morphism deduced
from the morphisms $p_{\Gamma_i}\colon \Omega_{\Gamma_i}\to X_{\Gamma_i}^\an$.

\subsection{Flat subvarieties}
%
Let $K$ be a complete extension of~$F$
and let $W$ be a closed analytic subspace of~$\Omega_K$.

The  following terminology is borrowed from the  analogous
notions in the differential geometry of hermitian symmetric domains.

We say that $W$ is \emph{irreducible algebraic} if there exists
a $K$-algebraic subvariety~$Y$ of~$(\P_1^n)_K$
such that $W$ is an irreducible component of 
the analytic space~$\Omega_K\cap Y^\an$.
In this case, one can take for~$Y$ the Zariski closure of~$W$
in~$(\P_1^n)_K$; it is irreducible and satisfies $\dim(Y)=\dim(W)$
(see~\cite{ducros2009}, proposition~4.22).

We say that $W$ is \emph{flat} if it can be defined by equations 
of the following form:
\begin{enumerate}
\item $z_i=c$, for some $i\in\{1,\dots,n\}$ and $c\in\Omega_{\Gamma_i}(K)$;
\item $z_j=g\cdot z_i$, for some pair $(i,j)$ of 
distinct
elements
of $\{1,\dots,n\}$ and $g\in\PGL(2,F)$.
\end{enumerate}
Assume that $W$ is flat and 
let $Y$ be the subvariety of~$(\P_1^n)_K$
defined by equations of this form 
which define~$W$ on~$\Omega_K$.
There exists 
a subset~$I$ of $\{1,\dots,n\}$ such that
the projection $q_I\colon \P_1^n\to \P_1^I$ given by the coordinates in~$I$
induces an isomorphism of~$Y$ to~$(\P_1^I)_K$.
This implies that $q_I$ induces an isomorphism from~$W$
to $\prod_{i\in I}\Omega_{i,K}$.
In particular, $W$ is irreducible,
even geometrically irreducible,
hence is irreducible algebraic.  Conversely, we observe
that if  $W$ is geometrically irreducible 
and if there exists a complete extension~$L$ of~$K$
such that $W_L$ is flat, then $W$ is flat.

We say that $W$ is \emph{geodesic} if, moreover, the
elements~$g$ in~(2) can be taken such that
$g\Gamma_i g^{-1}$ and $\Gamma_j$ 
are commensurable (i.e., their intersection 
has finite index in both of them).

Here is the main result of this paper.

\begin{theo}[Non-archimedean Ax-Lindemann theorem]\label{mt}
Let $F$ be a finite extension of $\Q_p$ and
let $(\Gamma_i)_{1\leq i\leq n}$ be a finite family
of  \emph{arithmetic} Schottky subgroups of $\PGL(2,F)$ of ranks~$\geq 2$.
As above,
let us set $\Omega=\prod_{i=1}^n \Omega_{\Gamma_i}$ and
$X=\prod_{i=1}^n X_{\Gamma_i}$, 
and let $p\colon \Omega \to X^\an$ be the morphism deduced
from the morphisms $p_{\Gamma_i}\colon \Omega_{\Gamma_i}\to X_{\Gamma_i}^\an$.

Let $V$ be an irreducible algebraic subvariety of~$X$
and let $W$ be an irreducible algebraic
subvariety of~$\Omega$, maximal among those contained in $p^{-1}(V^\an)$. 
Then every irreducible component of~$W_{\C_p}$ is flat.
\end{theo}

The proof of this theorem is given in section~\ref{sec.proof};
it follows the strategy of Pila--Zannier. 
In the archimedean setting, this strategy 
relies crucially on a theorem of Pila--Wilkie
about rational points on definable sets; we recall in 
section~\ref{sec.definability} the non-archimedean analogue of
this theorem, due to Cluckers, Comte and Loeser (see
\cite{cluckers-comte-loeser2015}), 
which is used here. 
It is at this point that 
we need the assumption that the group~$\Gamma$ be arithmetic.
This restriction is inherent to Pila--Zannier's strategy
and we do not know whether it can be bypassed.

In section~\ref{sec.schottky},
we recall the reader a few more facts on  $p$-adic Schottky groups
and $p$-adic uniformization,
essentially borrowed from the book~\cite{gerritzen-vanderput1980}.

In a final section, we prove a
characterization (theorem~\ref{theo.geodesic})
of geodesic subvarieties of~$\Omega$
as the geometrically irreducible algebraic subvarieties
whose projection to~$X$ is algebraic (``bi-algebraic subvarieties''),
in analogy with what happens in the context of Ax's theorem
or of Shimura varieties.

\section{The example of Shimura curves}
\label{sec.cherednik}

We begin by recalling the definition of Shimura curves and
their $p$-adic uniformization. The litterature  is unfortunately
rather scattered; we refer to~\cite{boutot-carayol1991}
for more detail, as well as to chapter~0 of~\cite{clark2003}.
 
\subsection{Complex Shimura curves}
Let $B$ be a quaternion division algebra with center~$\Q$;
we assume that it is indefinite, namely $B\otimes_\Q\R\simeq \MAT_2(\R)$.
Let then $\mathscr O_B$ be a maximal order of~$B$,
that is a maximal subring of~$B$ which is isomorphic to~$\Z^4$
as a $\Z$-module.
Let $H$ be the algebraic group of units of~$\mathscr O_B$, modulo center,
considered as a $\Z$-group scheme.
For every field~$R$ containing~$\Q$, one has 
$H(R)=(B\otimes_\Q R)^\times/Z((B\otimes_\Q R)^\times)$;
in particular, the group $H(\R)$ is isomorphic to~$\PGL(2,\R)$,
and we fix such an isomorphism.
Then the group $H(\R)$ acts by homographies 
on the double Poincaré upper half-plane
\[ \mathbf h^{\pm}=\C\setminus\R. \]
Let also $\Delta$ be a congruence subgroup of~$H(\Z)$;
recall that this means that there exists an integer~$N\geq1$
such that $\Delta$ contains the kernel of the canonical morphism
$H(\Z)\to H(\Z/N\Z)$. We assume that $\Delta$ has been chosen
small enough so that the stabilizer 
of every point of~$\mathbf h^{\pm}$  is trivial.
The quotient $\mathbf h^{\pm}/\Delta$ has a natural
structure of a compact Riemann surface and the projection
$p\colon \mathbf h^{\pm}\to \mathbf h^{\pm}/\Delta$
is an étale covering.

This curve parameterizes triples $(V,\iota,\nu)$,
where $V$ is a complex two dimensional abelian variety,
$\iota\colon \mathscr O_B\to \End(V)$ is a faithful action 
of~$\mathscr O_B$ on~$V$ and $\nu$ is a level structure ``of type~$\Delta$''
on~$V$. When   $\Delta$ is the kernel
of~$H(\Z)$ to~$H(\Z/N\Z)$, for some integer~$N\geq 1$,
such a level structure corresponds to an equivariant
isomorphism of~$V_N$, the subgroup of~$N$-torsion of~$V$, 
with~$\mathscr O_B/N\mathscr O_B$.

By~\cite{shimura1961}, it admits a canonical structure of an algebraic curve~$S$
which can be defined over a number field~$E$ in~$\C$.

\subsection{$p$-adic uniformization of Shimura curves}
Let $p$ be a prime number at which~$B$ ramifies, which means
that $B\otimes_\Q\Q_p$~is a division algebra.
Let also~$F$ be the completion of the field~$E$
at a place dividing~$p$; we denote by~$\C_p$ the $p$-adic completion
of an algebraic closure of~$F$. We still denote by~$S$
the $F$-curve deduced from an $E$-model of the complex curve~$S$.

Let $\Omega=(\P_1)_F^\an\setminus \P_1(\Q_p)$ be the extension
of scalars to~$F$ of Drinfeld's upper half plane.
According to the theorem of Cherednik-Drinfeld
(\cite{cerednik1976,drinfeld1976}; see also~\cite{boutot-carayol1991}
for a detailed exposition), and up to replacing~$F$ by a finite
unramified extension, 
the $F$-analytic curve~$S^\an$ admits 
a ``$p$-adic uniformization'' which takes the form of a surjective
analytic morphism
\[ j \colon \Omega \to S^\an ,  \]
which identifies $S^\an$ with
the 
quotient of~$\Omega$ by the action
of a subgroup~$\Gamma$ of~$\PGL(2,\Q_p)$.
Up to replacing~$\Delta$ by a smaller congruence subgroup,
which replaces~$S$ by a finite (possibly ramified) covering, 
we may also assume that $\Gamma$ is
a $p$-adic Schottky subgroup acting freely on~$\Omega$,
and that $j$ is topologically étale. Then the 
morphism~$j\colon \Omega\to S^\an$ is the universal cover of~$S^\an$.

Let us describe this subgroup.
Let $A$ be the quaternion division algebra
over~$\Q$ with the same invariants as~$B$,
except for those invariants at~$p$ and~$\infty$ which are switched.
In particular, $A\otimes_\Q\R$ is Hamilton's quaternion
algebra, while $A\otimes_\Q\Q_p\simeq \MAT_2(\Q_p)$.
Let $G$ be the algebraic group of units of~$A$, modulo center;
in particular, $G(\Q_p)\simeq\PGL(2,\Q_p)$.
As explained in~\cite{boutot-carayol1991},
the discrete subgroup $ \Gamma$ is the intersection 
of $G(\Q)$ with a compact open subgroup of 
$G(\mathbf A_\mathrm f)$,
the adelic group associated with~$G$
where the place at~$\infty$ is omitted.

\begin{lemm}\label{prop.fundamental-drinfeld}
The group~$\Gamma$  is conjugate 
to an arithmetic Schottky subgroup in~$\PGL(2,\Q_p)$, 
its rank is at least~$2$,
and its limit set is equal to~$\P_1(\Q_p)$.
\end{lemm}
\begin{proof}
The group~$\Gamma$ is a discrete subgroup of $\PGL(2,\Q_p)$
hence  its limit set $\mathscr L_\Gamma$ 
is a $\Gamma$-invariant subset of $\P_1(\Q_p)$.
In other words,
the Drinfeld upper half-plane $\Omega=\P_1^\an\setminus \P_1(\Q_p)$ is an open subset
of $\Omega_\Gamma=\P_1^\an\setminus\mathscr L_\Gamma$.
By the theory of Mumford curves and Schottky groups,
see~\cite{gerritzen-vanderput1980},
the analytic curve $(\P_1^\an\setminus\mathscr L_\Gamma)/\Gamma$
is algebraic, and admits the analytic curve~$S^\an=\Omega/\Gamma$ 
as an open subset.
According to the Cherednik-Drinfeld theorem, 
the curve~$S^\an$ is projective. 
This implies that $\Omega=\P_1^\an\setminus\mathscr  L_\Gamma$,
hence $\mathscr L_\Gamma=\P_1(\Q_p)$.

After base change to~$\Q_p$,
the algebraic $\Q$-group~$G$ becomes isomorphic to $\PGL(2)_{\Q_p}$.
Consequently, there exists a finite algebraic extension~$K$
of~$\Q$, contained in~$\Q_p$, such that $G_K\simeq \PGL(2)_K$.
By such an isomorphism, $G(\Q)$ is mapped into~$\PGL(2,K)$;
this implies that the group~$\Gamma$ is conjugate to an arithmetic
group.

Since $\Gamma$ is a Schottky group, it is free. Since it is non-abelian,
its rank is at least~$2$.
\end{proof}

By this lemma, the following result is a special case of
our main theorem~\ref{mt}.
\begin{theo}\label{mt-shimura}
Let $F$ be a finite extension of $\Q_p$, 
let $\Omega=(\P_1)_F^\an\setminus\P_1(\Q_p)$ and
let $j\colon \Omega^n \to S^\an$ be the  Cherednik--Drinfeld
uniformization of a product of Shimura curves.
Let $V$ be an irreducible algebraic subvariety of~$S$
and let $W\subset\Omega^n$ 
a maximal irreducible algebraic
subvariety of~$j^{-1}(V^\an)$. 

Then every irreducible component of~$W_{\C_p}$ is flat.
\end{theo}

\subsection{}
By the same arguments, one can show that our main theorem~\ref{mt}
also applies to the uniformizations of Shimura curves associated
with quaternion division algebras over totally real fields,
as considered by Cherednik~\cite{cerednik1976}
and Boutot--Zink~\cite{boutot-zink1995}.

\subsection{}
As suggested by J.~Pila and explained to us by Y.~André, 
theorem~\ref{mt-shimura} can also be deduced from its complex analogue,
which is a particular case of~\cite{ullmo-yafaev2014}.
The crucial ingredient is a deep theorem of André (\cite{andre2003}, III, 4.7.4)
stating that the $p$-adic uniformization and the complex
uniformization of Shimura curves satisfy the \emph{same} 
non-linear differential equation.
His proof relies on a delicate description of the Gauss-Manin equation 
in terms of convergent crystals and on the tempered fundamental group
introduced by him.
From that point on, one can apply Seidenberg's
embedding theorem~\cite{seidenberg1958}
in differential algebra 
to prove that both the complex and non-archimedean Ax-Lindemann  
theorems are equivalent to a single statement in differential algebra,
in the original spirit of Ax's paper~\cite{ax:1971}.

\section{Definability --- A \texorpdfstring{$p$}{p}-adic Pila-Wilkie theorem}
\label{sec.definability}

\subsection{}\label{ss.subanalytic}
There are two distinct notions of $p$-adic analytic geometry,
one is  ``naïve'', and the other one is rigid  analytic. 
(Regarding rigid analytic geometry, we shall work in the framework
defined by Berkovich.)
These two notions give rise to three classes of sets,
and we shall use them all in this paper.
Let $F$ be a finite extension of~$\Q_p$.

\begin{enumerate}\def\theenumi{\alph{enumi}}\def\labelenumi{\theenumi)}
\item
\emph{Semialgebraic} and \emph{subanalytic} subsets of $\Q_p^n$
are defined by Denef and van den Dries
in~\cite{denef-vandendries1988}; 
see also \cite[p.~26]{cluckers-comte-loeser2015}.

Replacing~$\Q_p$ by a finite extension~$F$, this leads to an analogous
notion of $F$-semialgebraic, or $F$-subanalytic, subset of~$F^n$.
Considering affine charts, one then defines
$F$-semialgebraic or $F$-subanalytic subsets of~$V(F)$,
for every (quasi-projective, say) algebraic variety~$V$ defined over~$F$.

On the other hand, the Weil restriction functor
assigns to~$V$
an algebraic variety~$W$ defined over~$\Q_p$ together
with a canonical identification $V(F)\to W(\Q_p)$;
we shall say that a subset of $V(F)$ is $\Q_p$-semialgebraic,
resp. $\Q_p$-subanalytic, if its image in~$W(\Q_p)$
is $\Q_p$-semi-algebraic, resp. $\Q_p$-subanalytic.
Observe that $F$-semi-algebraic subsets of~$V(F)$ are $\Q_p$-semi-algebraic,
and that $F$-subanalytic subsets of~$V(F)$ are $\Q_p$-subanalytic.

Recall that a $F$-subanalytic subset~$S$ is said to be smooth 
of dimension~$d$ at a point~$x$
if it possesses a neighborhood~$U$ which is isomorphic to  
the unit ball of~$F^d$; then $S$ is smooth of dimension~$d$ 
at every point of~$U$.

\item In~\cite{lipshitz1993}, Lipshitz defined a notion of 
\emph{rigid subanalytic subset} of $\C_p^n$. We shall use  in this paper
the variant (\cite{lipshitz-robinson2000b}, definition~2.1.1)
where the coefficients of all polynomials and power series involved belong 
to~$F$; we will call them \emph{rigid $F$-subanalytic.}
The notion extends to subsets of~$V(\C_p)$, where $V$ is an
algebraic variety defined over~$F$.
\end{enumerate}

These classes of sets are stable under boolean operations and projections
(corollary~4.3 of \cite{lipshitz-robinson2000c}),
admit cell decompositions 
(theorem~7.4 of~\cite{cluckers-lipshitz-robinson2006}),
a natural notion of dimension
(in fact, they are b-minimal in the sense of~\cite{cluckers-loeser2007}),
as well as a natural notion of smoothness.

\begin{lemm}\label{lemm.remark}
Let $F$ be a finite extension of~$\Q_p$ contained in~$\C_p$
and let $V$ be an algebraic variety over~$F$.
Let $Z$ be a rigid $F$-subanalytic subset of~$V(\C_p)$.
Then $Z(F)=Z\cap V(F)$ is an $F$-subanalytic subset of~$V(F)$.
\end{lemm}
\begin{proof}
We may assume that $V=\mathbf A^n$.
Then $Z$ can be defined by a quantifier-free formula of 
the above-mentioned variant of Lipshitz's analytic language,
and our claim follows from the very definition of this language.
\end{proof}


\subsection{}
A \emph{block} in $\Q_p^n$ is either empty, or a singleton, or a smooth
subanalytic subset of pure dimension~$d>0$ which is contained in
a smooth semialgebraic subset of dimension~$d$.

A \emph{family of blocks} 
in $\Q_p^n\times\Q_p^s$ is a subanalytic subset~$W$
such that there exists an integer $t\geq0$
and a semialgebraic set $Z\subset \Q_p^n\times\Q_p^t$
such that for every $\sigma\in\Q_p^s$, there exists
$\tau\in\Q_p^t$ such that the fibers $W_\sigma$ and $ Z_\tau$ 
are smooth of the same dimension, and $W_\sigma\subset Z_\tau$.
(In particular, the sets $W_\sigma$, for $\sigma\in\Q_p^s$, 
are blocks in $\Q_p^n$.)

Let $F$ be a finite extension of~$\Q_p$.
Considering Weil restriction, we deduce from these notions
the definition of a block in~$F^n$,
or of a family of blocks in~$F^n\times\Q_p^t$.

%

\subsection{}\label{ss.heights}
Let $H$ be the standard height function  on~$\overline\Q$;
for $x\in\Q$, written as a fraction $a/b$ in lowest terms, one has
$H(x)=\max(\abs a,\abs b)$. We also write $H$ for the height function
on~$\overline \Q^n$ defined by $H(x_1,\dots,x_n)=\max_i(H(x_i))$. 
Viewing $\GL(d,\overline\Q)$ as a subspace of $\overline\Q^{d^2}$,
it defines a height function on $\GL(d,\overline\Q)$. 
There exists a strictly positive real number~$c$
such that $H(gg')\leq c H(g)H(g')$ 
for every $g,g'\in\GL(d,\overline\Q)$,
and $H(g^{-1})\ll H(g)^c$ for every $g\in\GL(d,\overline\Q)$.
When $d=2$ and $g\in\SL(2,\overline\Q)$, one even has $H(g^{-1})=H(g)$.

Let $g\in\GL(d,\overline\Q)$.
If $g$ is diagonal, then $H(g^n)=H(g)^n$ for every $n\in\Z$.
More generally, if  $g$ is \emph{semisimple}, we have  upper and lower bounds
$H(g)^n\ll H(g^n)\ll H(g)^n$, for every $n\in\Z$.

By abuse of language, if $G$ is a linear algebraic $\overline\Q$-group, we 
implicitly choose an embedding in some linear group,
which furnishes a height function~$H$ on $G(\overline\Q)$. 

The actual
choice of this height function depends on the chosen embedding,
but any other height function~$H'$ is equivalent, in the sense
that there is a strictly
positive real number~$c$ such that $H(x)^{1/c}\ll H'(x)\ll H(x)^c$
for every $x\in G(\overline\Q)$.

\subsection{}\label{subsec.many}
Let $Z$ be a subset of~$F^n$ and let $K$ be finite extension
of~$\Q$ contained in~$F$. 
We write $Z(K)=Z\cap K^n$ ($K$-rational points of~$Z$).
For every real number~$T$, 
we define $Z(K;T)=\{x \in Z(K)\sozat H(x)\leq T\}$;
for every integer~$D$, we also define $Z(D;T)$
to be the set of points $x\in Z(F)$ such that $[\Q(x_i):\Q]\leq D$
for every $i\in\{1,\dots,n\}$
and $H(x)\leq T$. These are finite sets.

We say that $Z$ has \emph{many $K$-rational points}
if there exist strictly positive real numbers~$c,\alpha$ such that 
\[ \Card \big( Z(K;T)\big)  \geq c T^\alpha \]
for all $T$ large enough.
This notion only depends on the equivalence class of the height.

\subsection{}
In~\cite{cluckers-comte-loeser2015},
Cluckers, Comte and Loeser established
a $p$-adic analogue of a theorem of Pila-Wilkie~\cite{pila-wilkie2006}
concerning the rational points of a definable set.
We will use the following variant
of \cite[Theorem~4.2.3]{cluckers-comte-loeser2015}.
\begin{theo}\label{theo.pw}
Let $F$ be a finite extension of~$\Q_p$ and let $K$ 
be a finite extension of~$\Q$, contained in~$F$.
Let $Z\subset F^n$ be a $\Q_p$-subanalytic subset.
Let $\eps>0$.
There exist $s\in\N$,  $c\in\R$ and a family of blocks
$W\subset Z\times\Q_p^s$ satisfying the following property:
for every $T>1$, there exists a subset $S_T\subset\Q_p^s$
of cardinality $< cT^\eps$ such that
$Z(K;T)\subset \bigcup_{\sigma \in S_T} W_\sigma$.
\end{theo}
\begin{proof}
Let $d=[F:\Q_p]$. By Krasner's lemma, there exists
an algebraic number~$e\in F$ of degree~$d$ such that $F=\Q_p(e)$.
Then the basis $(1,e,\dots,e^{d-1})$ defines 
a $\Q_p$-linear bijection $\psi\colon \Q_p^d\xrightarrow\sim F$, 
$(x_1,\dots,x_d)\mapsto \sum x_i e^{i-1}$. 
Let $\phi\colon F\simeq\Q_p^d$ be its inverse.

By construction, if $K$ is a number field
contained in~$F$ and $x\in K^d$, then $\psi(x)\in K(e)$;
in particular, $[\Q(\psi(x)):\Q]\leq d [\Q(x):\Q]$.
Conversely, if $x\in K$, then the coordinates of~$\phi(x)$ in~$\Q_p^{d}$
belong to the Galois closure~$K(e)'$ of
the compositum~{$K\cdot \Q(e)$}
hence are algebraic numbers, 
of degrees $\leq D=[K(e)':\Q]$.
In other words, $\phi$ and $\psi$ induce bijections
at the level of algebraic points. Since these maps are linear,
there exists a positive real number~$a>0$
such that $a^{-1} H(x)\leq H(\phi(x))\leq a H(x)$
for every $x\in K$.

We deduce from $\phi$
a $\Q_p$-linear isomorphism $\phi\colon F^n\to\Q_p^{nd}$.
In particular, $Z'=\phi(Z)$ 
is a subanalytic subset of~$\Q_p^{nd}$.
The morphism~$\phi$ maps algebraic points of given degree to algebraic points
of uniformly bounded degree, and there exists a positive real number~$a>0$
such that $a^{-1} H(x)\leq H(\phi(x))\leq a H(x)$
for every $x\in Z(K)$.

The definition of a family of blocks that we have
adopted here is slightly stronger
than the one used in Theorem~4.2.3 of~\cite{cluckers-comte-loeser2015}.
However, all proofs go over without any modification,
so that there exists a family of blocks~$W'\subset Z'\times\Q_p^s$
such that for any $T>1$, there exists a subset $S_T\subset\Q_p^s$
of cardinality~$<cT^\eps$ 
such that $Z'(D;T)\subset \bigcup_{\sigma\in S_T}W'_\sigma$.
Let $\psi\colon F^n\times\Q_p^s\to \Q_p^{nd}\times\Q_p^s$ be the map
$(x,y)\mapsto (\phi(x),y)$ and
let $W=\psi^{-1}(W')\subset F^n\times\Q_p^s$.
By definition, $W$ is a family of blocks in~$Z$.
Moreover, for any $T>1$, one has 
\[ Z(F;T)\subset \psi^{-1}(Z'(D;aT))\subset \bigcup_{\sigma\in S_{aT}}
 \phi^{-1}(W'_\sigma) = \bigcup_{\sigma\in S_{aT}}
 W_\sigma. \]
Since $\Card(S_{aT})\leq ca^\eps T^\eps$, the family of blocks~$W$
satisfies the requirements of the theorem.
\end{proof}

\section{Zariski closures and analytic functions}

\subsection{}
Let $F$ be a complete non-archimedean valued field.
Let $V$ be  an $F$-scheme of finite type.
One says that a subset $K$ of~$V^\an$ is  \emph{sparse}
if there exist a set~$T$ and  a 
subset~$Z$ of $V^\an\times T$ such that
for every $t\in T$, $Z_t=\{x\in V^\an\sozat (x,t)\in Z\}$
is a Zariski-closed subset of~$V^\an$ with empty interior,
and $K= \bigcup_{t\in T} Z_t$. 

\begin{lemm}\label{lemm.sparse-empty-interior}
A sparse set has empty interior.
\end{lemm}
\begin{proof}
Let us say that a point~$x\in V^\an$ is maximally Abhyankar 
if the rational rank of the value group of~$\mathscr H(x)$ is equal to
$\dim_x(V^\an)$.
If $V$ is irreducible, then maximally Abhyankar points are dense in~$V^\an$;
moreover, each of them is Zariski dense.
Let $K$ be a sparse set in~$V^\an$; write $K=\bigcup_t Z_t$ as above.
Let us argue by contradiction and 
let $U$ be a non-empty subset of~$V^\an$ contained in~$K$.
By what precedes, there exists a maximally Abhyankar point $x\in U$.
Let $t\in T$ be such that $x\in Z_t$. Then $Z_t$ contains the
Zariski closure of~$x$ in~$V^\an$, so that $Z_t$ contains an irreducible
component of~$V^\an$, contradicting the definition of a sparse set.
\end{proof}

\begin{lemm}\label{lemm.sparse-ext}
Let $F'$ be an algebraically closed complete extension of~$F$,
let $q\colon V_{F'}^\an\to V^\an$ be the base change morphism.
Let $K$ be a closed sparse subset of~$V^\an$ 
and let $K'=q^{-1}(K)$. Then $K'$ is sparse.
\end{lemm}
\begin{proof}
Indeed, if $K=\bigcup_{t\in T}Z_t^\an$ is a
description of the sparse set~$K$,
then $K'=\bigcup_{t\in T} (Z_t)_{F'}^\an $
shows that $K'$ is sparse as well.
\end{proof}

\begin{lemm}\label{lemm.curve-sparse}
Let us assume that $K$ is sparse, and 
let $C\subset V$ be a geometrically irreducible curve 
such that $C^\an\not\subset K$.
Then $C^\an\setminus K$ is connected.
\end{lemm}
\begin{proof}
Using lemma~\ref{lemm.sparse-ext}, we reduce to the case where 
$F$ is algebraically closed; we may moreover assume
that $C$ is reduced.
Let $K=\bigcup_{t\in T} Z_t^\an$ be a description of~$K$ as above. 
Up to adding the singular locus of~$C$ to~$K$, 
we may assume that $C$ is smooth.
By assumption, for every $t\in T$, $C\not\subset Z_t^\an$;
consequently, $Z_t^\an\cap C^\an$ consists of rigid points of~$C^\an$,
hence $K\cap C^\an$ consists of rigid points of~$C^\an$.
In the topological description of 
smooth geometrically irreducible analytic curves 
as real graphs (\cite{berkovich1990}, chapter~4),
their rigid points are endpoints, 
so that $C^\an\setminus (K\cap C^\an)$  is connected as well.
\end{proof}
\begin{prop}\label{lemm.irred}
Let $F$  be a complete non-archimedean valued  field.
Let $V$ be an $F$-scheme of finite type which is
geometrically connected (resp. geometrically irreducible)
and let $K$ be a closed sparse subset of $V^\an$.
Then $V^\an\setminus K$ is a geometrically connected 
(resp. geometrically irreducible) analytic space.
\end{prop}
The particular case $K=\emptyset$ implies the ``\textsc{gaga}''-type consequence
that if $V$ is geometrically connected (resp. geometrically irreducible),
then so is $V^\an$. 
\begin{proof}
Using lemma~\ref{lemm.sparse-ext}, we reduce to the
case where $F$ is algebraically closed.
By assumption, $V$ is connected.
Let us prove that $V^\an\setminus K$ is connected.
Let $x,y\in V^\an\setminus K$. 
Let $F'$ an algebraically closed complete valued
field containing both~$\mathscr H(x)$ and~$\mathscr H(y)$,
and view $x,y$ as elements of~$V(F')$;
let $q\colon V_{F'}^\an\to V^\an$ be the base change morphism
and let $K'=q^{-1}(K)$; by lemma~\ref{lemm.sparse-ext},
this is a sparse subset of~$V_{F'}^\an$.
By \cite[p. 56]{mumford74}, 
there exists an irreducible curve $C\subset V_{F'}$
which passes through~$x$ and~$y$. Then $C^\an$ is connected.
One has $C\not\subset K'$, by definition of~$K'$;
it follows from lemma~\ref{lemm.curve-sparse}
that $C^\an \setminus (K'\cap C^\an)$ is connected.
Consequently,
 $x$ and~$y$ belong to the same component of~$V_{F'}^\an\setminus K'$,
hence their images in~$V^\an\setminus K$ belong
to the same connected component.
This proves that $V^\an\setminus K$ is connected.
%

Let us now assume that $V$ is geometrically irreducible.
The normalization morphism $p\colon W\to V$ is finite, 
and 
$W$ is geometrically connected.  
Since $p^{-1}(K)$ is a sparse subset of $W^\an$,
it follows from the first part
of the lemma that $W^\an\setminus p^{-1}(K)$ is geometrically connected.
Since $W^\an$ is the normalization of $V^\an$
(\cite{ducros2016}, lemma~2.7.15),
then
$W^\an\setminus p^{-1}(K)=p^{-1}(V^\an\setminus K)$ is the normalization
of $V^\an\setminus K$. 
By theorem~5.17 of~\cite{ducros2009}, 
this implies that $V^\an\setminus K$ is geometrically irreducible.
\end{proof}

\begin{coro}\label{coro.irr-comp-1}
Let $F$ be a complete valued field, let $V$ be an $F$-scheme of finite
type and let~$K$ be a closed sparse subset of~$V^\an$.
The set of irreducible components of~$V^\an\setminus K$ is finite.
If $V$ is equidimensional, then each of them has dimension~$\dim(V)$.
\end{coro}
\begin{proof}
We may assume that $V$ is irreducible.  Let $\Omega=V^\an\setminus K$.
Let $E$ be the completion of an algebraic closure of~$F$. 
By proposition~\ref{lemm.irred},
$\Omega_E \cap Z^\an$ is  irreducible, for every irreducible
component~$Z$ of~$V_E$, and the family of these intersections
is the family of irreducible components of~$\Omega_E$.
The finiteness statement then follows from~\cite[lemme~4.25]{ducros2009},
while the one about dimension follows
from proposition~4.22 of~\cite{ducros2009}.
\end{proof}

\begin{coro}\label{coro.irr-comp-2}
Let $F$ be a complete valued field, let $V$ be an 
irreducible $F$-scheme of finite
type and let~$K$ be a closed sparse subset of~$V^\an$.
Let $W$ be an irreducible component of $V^\an\setminus K$.
If $W$ is geometrically irreducible,
then $V$ is geometrically irreducible as well, 
one has $W=V^\an\setminus K$, and $W$ is topologically dense in~$V^\an$.
\end{coro}
\begin{proof}
Let $E$ be a complete algebraically closed extension of~$F$,
let $V_1,\dots,V_n$ be the irreducible components of~$V_E$;
let $L$ be the preimage of~$K$ in~$V_E$; 
it is a closed sparse subset of~$V_E^\an$ (lemma~\ref{lemm.sparse-ext}).
Consequently, $L_j=V_j^\an\cap L$ is a closed sparse subset
of~$V_j^\an$, for every~$j$. By proposition~\ref{lemm.irred},
$W_j=V_j^\an\setminus L_j$ is geometrically irreducible.
The automorphism group $\Aut(E/F)$  acts transitively
on the set $\{V_1,\dots,V_n\}$ of irreducible components  of~$V_E$,
hence on the set $\{W_1,\dots,W_n\}$ of irreducible components
of~$V_E^\an\setminus L$. 
Since $V_E$ is geometrically irreducible, there exists an index~$j$
such that $W_E=W_j$; then $\Aut(E/F)$ fixes~$W_j$, so that $n=1$ and $j=1$.
This proves that $V$ is geometrically irreducible.
By proposition~\ref{lemm.irred}, one has $W=V^\an\setminus K$.
By lemma~\ref{lemm.sparse-empty-interior}, $W$ is topologically dense
in~$V^\an$. 
\end{proof}

\begin{prop}\label{prop.closure-analytic}
Let~$F$ be a finite extension of $\Q_p$.
Let~$A$ be an affine scheme of finite type over~$F$
and let $\Omega\subset A^\an$ be the complement
of a closed sparse subset.
Let~$X$ be a closed analytic subspace of~$\Omega$.
Let~$V$ be a $\Q_p$-semi-algebraic subset of~$A(F)$, contained in~$X(F)$,
and let~$W$ be its Zariski closure in~$A$.

One has  $W^\an\cap\Omega \subset X$.
\end{prop}
\begin{proof}
The following proof is inspired by that of lemma~4.1 
in~\cite{pila-tsimerman2013}.

We argue by noetherian induction on~$W$, assuming that
if $W'$ is the Zariski closure of a $\Q_p$-semi-algebraic
subset~$V'$ of~$A(F)$ contained in~$X(F)$,  and if $W'\subsetneq W$,
then $(W')^\an\cap\Omega\subset X$.

First assume that $W$ is not irreducible.
Then any irreducible component~$W'$ of~$W$
is the Zariski closure in~$A$ of 
$V\cap W'(F)$, a $\Q_p$-semi-algebraic subset of~$A(F)$;
by induction, $(W')^\an\cap\Omega\subset X$,
so that $W^\an\cap \Omega\subset X$.

We may thus assume that $W$ is irreducible;
since its subset~$W(F)$ of $F$-rational points contains~$V$,
it is Zariski-dense in~$W$, so that $W$ is geometrically irreducible.

Let $K= A^\an\setminus \Omega$; by assumption, $K$ is closed and sparse;
let $K=\bigcup S_t ^\an$ be a presentation of~$K$, where for
every~$t$, $S_t$ is a Zariski-closed subset with empty interior of~$A$.
Since $W$ is irreducible and not contained in~$S_t$,
$W\cap S_t$ is a strict Zariski-closed subset of~$W$.
Consequently, $W^\an\cap K$ is a sparse subset of~$W^\an$.
By proposition~\ref{lemm.irred}, $W^\an\cap\Omega$ 
is thus a geometrically irreducible analytic space.


Let $\mathrm R$ be the Weil restriction functor from~$F$ to~$\Q_p$.
By definition, $A(F)$ is identified with $\mathrm R(A)(\Q_p)$ and
we shall write~$\mathrm R(V)$ for the image of~$V$ inside $\mathrm R(A)(\Q_p)$.
Let then $Z$ be the Zariski closure of $\mathrm R(V)$ inside $\mathrm R(A)$.

Let $Z'$ be an irreducible component of~$Z$. 
Then $Z'\cap \mathrm R(V)$ is a semi-algebraic subset of~$\mathrm R(A)$,
of the form $\mathrm R(V')$, for a unique $\Q_p$-semi-algebraic subset~$V'$ 
of~$V$. 
When~$Z'$ varies, the corresponding subsets~$V'$ cover~$V$;
we may thus choose~$Z'$ such that $V'$ is Zariski dense in~$W$.
Replacing~$V$ by~$V'$, we may assume that $Z$ is irreducible;
then it is geometrically irreducible, because
its set of $\Q_p$-points is Zariski dense.

Since $V$ is $\Q_p$-semialgebraic, the subset $\mathrm R(V)$ 
of $\mathrm R(A)(\Q_p)$ is semialgebraic,
hence the dimension of~$Z$ coincides with the dimension of~$V$
as a $\Q_p$-semialgebraic subset of~$A(F)$.
Consequently, $\dim_{\mathrm{Zar}}(Z)=\dim (Z(\Q_p))=\dim (\mathrm R(V))$.

Since $W$ is a Zariski closed subset of~$A$ containing~$V$,
the subscheme~$\mathrm R(W)$ is Zariski closed in~$\mathrm R(A)$
and contains~$\mathrm R(V)$, so that $Z\subset \mathrm R(W)$.
By Weil restriction, 
the inclusion $Z\to\mathrm R(W)$
corresponds to a morphism $g\colon Z_F\to W$.
Let $x\in A(F)$ and let $\tilde x\in \mathrm R(A)(\Q_p)$ be the corresponding
point; if $x\in V$, then $\tilde x\in \mathrm R(V)\subset Z(\Q_p)$,
hence $\tilde x\in Z_F(F)$. By the definition of the Weil restriction
functor, one has $g(\tilde x)=x$.
In particular, the image of~$Z_F(F)$ under~$g$ contains~$V$, 
hence $g$ is dominant, by definition of~$W$.

The morphism~$g$ induces an analytic morphism 
$g^\an\colon Z_F^\an \to W^\an\subset A^\an$.
The inverse image of~$W^\an\cap\Omega$ is the complement of
a closed sparse subset of~$Z_F^\an$; 
since $Z_F^\an$ is geometrically irreducible,
corollary~\ref{coro.irr-comp-1}
implies that
$(g^\an)^{-1}(W^\an\cap\Omega)$ is geometrically irreducible,
of dimension~$\dim(Z_F^\an)$.
Let $Y=(g^\an)^{-1}(W^\an\cap X)$; it is a Zariski closed
analytic subset of $(g^\an)^{-1}(W^\an\cap\Omega)$.

Let us admit for a moment that $\dim(Y)=\dim(Z_F)$
and let us conclude 
that $W^\an\cap\Omega\subset X$.
Since $\dim(Z_F^\an)=\dim(Z_F)=\dim((g^\an)^{-1}(W^\an\cap\Omega))$,
we see that 
\[ Y= (g^\an)^{-1}(W^\an\cap X)=(g^\an)^{-1}(W^\an\cap\Omega).\]
The morphism $g\colon Z_F\to W$ being dominant, 
its image contains a non-empty open subset~$W'$ of~$W$.
Since $W$ is geometrically irreducible, $(W')^\an$ is dense
in~$W^\an$; in particular, the image of $g^\an$ meets
any non-empty open subset of~$W^\an$.
Since $(g^\an)^{-1}(W^\an\cap (\Omega\setminus X))$ is empty,
by the preceding equality,
this implies that $W^\an\cap(\Omega\setminus X)$ is empty,
hence $W^\an\cap\Omega=W^\an\cap X$.

It remains to prove the equality $\dim(Y)=\dim(Z_F)$.

Let us consider a semi-algebraic 
cell decomposition of~$\mathrm R(A)(\Q_p)$ 
which is adapted to $\mathrm R(V)$, $Z(\Q_p)$, $Z_\sing(\Q_p)$,
and to their singular loci:
a finite partition of~$\mathrm R(A)(\Q_p)$ into ``open cells''
such that these $\Q_p$-semi-algebraic subsets
are unions of cells
(\cite{denef1986}, see also~\cite{cluckers-loeser2007}).

Let $\tilde C$ be a cell of dimension $\dim(\mathrm R(V))$ 
which is contained in~$\mathrm R(V)$.
Since $\dim(Z_\sing(\Q_p))\leq \dim(Z_\sing)<\dim(Z)=\dim(\mathrm R(V))$,
the cell~$\tilde C$ is disjoint from~$Z_\sing(\Q_p)$.
By definition
of a cellular decomposition,  $\tilde C$ is open
in~$\mathrm R(V)$ and in $(Z\setminus Z_\sing)(\Q_p)$.

Let $C$ be the subset of~$V$ corresponding to~$\tilde C$.
Since the identification of~$C$ with~$\tilde C$ provided
by the Weil restriction functor is a homeomorphism
which respects the singular loci,
$C$ is an open subset of~$V$.

Let $x$ be a point  of~$C$ and let $\tilde x$ be the corresponding
point of $\tilde C$.
By what precedes, $\mathrm R(V)$, $Z(\Q_p)$ and~$Z$ are smooth at~$\tilde x$,
so that 
$\mathrm T_{\tilde x}(\mathrm R(V))=\mathrm T_{\tilde x} (Z(\Q_p))
= \mathrm T_{\tilde x} (Z)$.
In particular, these three $\Q_p$-vector spaces have the same
dimension, equal to $\dim(T_x (V))=\dim(V)$.

Since $g(\tilde x)=x\in X$, one has $\tilde x\in Y$; 
more generally, $\tilde C\subset Y$.
The tangent space $\mathrm T_{\tilde x} (Y)$ of~$Y$ at~$\tilde x$ 
is an $F$-vector subspace
of $\mathrm T_{\tilde x} (Z_F)=(\mathrm T_{\tilde x} (Z))_F$
which contains $\mathrm T_{\tilde x}(\tilde C)=\mathrm T_{\tilde x}(Z)$.
Consequently, $\mathrm T_{\tilde x}(Y)=\mathrm T_{\tilde x}(Z_F)$.
This implies that the analytic space~$Y$ has dimension~$\dim(Z_F)$,
and concludes the proof.
%
\end{proof}

\section{Complements on \texorpdfstring{$p$}{p}-adic Schottky groups and uniformization}
\label{sec.schottky}

Let $F$ be a finite extension of~$\Q_p$. Unless precised
otherwise, analytic spaces are $F$-analytic spaces.

\subsection{}
Let $a\in F$ and $r\in\R_{>0}$; as usual, we let
$B(a,r)$ and $E(a,r)$ be the subsets of $(\mathbf A^1)^\an$
of points $x$ such that $\abs{T(x)-a}<r$ and $\abs{T(x)-a}\leq r$ respectively.
The subspace~$B(a,r)$ is called a \emph{bounded open disk}; 
we say that $E(a,r)$ is the corresponding \emph{bounded closed disk}.
If $B$ is a bounded open disk, we write $B^+$ for the corresponding
bounded closed disk.
We will say that such a disk is strict if its \emph{radius}~$r$
belongs to $\abs{F^\times}^\Q$.

To these disks, we also add the unbounded open disks $\P_1^\an\setminus E(a,r)$
and the unbounded closed disks $\P_1^\an\setminus B(a,r)$.
An unbounded disk is said to be strict if its complementary disk
is strict.

The image by an homography $\gamma\in\PGL(2,F)$
of an open (resp. closed, resp. strict) disk is again an open
(resp. closed, resp. strict) disk.

\subsection{}\label{ss.delta}
We endow $\P_1(\C_p)$ with the distance given by
\[ \delta(x,y) = \frac{\abs{x-y}}{\max(1,\abs x)\max(1,\abs y)} \]
for $x,y\in\C_p$ --- 
it is invariant under the action of $\PGL(2,\mathscr O_{\C_p})$. 
Moreover, an elementary calculation
shows that every element $g\in\PGL(2,\C_p)$ is Lipschitz for this distance
(see also Thm~1.1.1 of~\cite{rumely1989}).


\subsection{}\label{ss.lipschitz}
Let $\Gamma$ be a Schottky group in~$\PGL(2,F)$,
let $\mathscr L_\Gamma\subset \P_1(F)$ be its limit set 
and $\Omega_\Gamma=\P_1^\an\setminus \mathscr L_\Gamma$.
For any rigid point $x\in \Omega_\Gamma$, 
let $\delta_\Gamma(x)$ be the $\delta$-distance of~$x$ to~$\mathscr L_\Gamma$.

For every $\gamma\in\PGL(2,F)$,
there exists a real number~$c\geq 1$ such that 
$c^{-1}\delta_\Gamma(z)\leq
   \delta_\Gamma(\gamma\cdot z)\leq c \delta_\Gamma(z)$ 
for every rigid point~$z\in\Omega_\Gamma$.

\begin{lemm}\label{lemm.minor}
Let $\mathfrak G$ be a compact subset of~$\Omega_\Gamma$. There exists a strictly positive
real number~$c$ such that $\delta_\Gamma(x)\geq c$ for every rigid point~$x\in\mathfrak G$.
\end{lemm}
\begin{proof}
Arguing by contradiction, we assume that there exists a sequence $(x_n)$ of rigid
points of~$\mathfrak G$ such that $\delta_\Gamma(x_n)\to 0$.
For every~$n$, let $\xi_n\in\mathscr L_\Gamma$ such that $\delta_\Gamma(x_n)=\delta(x_n,\xi_n)$;
it exists since~$\mathscr L_\Gamma$ is compact. 
Extracting a subsequence if necessary, we assume that the sequence~$(\xi_n)$
converges to a point~$\xi$ of~$\mathscr L_\Gamma$. Then $\delta(x_n,\xi)\to 0$.
This implies that the sequence~$(x_n)$ converges to~$\xi$ in the Berkovich space~$\P_1^\an$.
Since $\mathfrak G$ is compact, one has $\xi\in\mathfrak G$, a contradiction.
\end{proof}

%
 
\subsection{}\label{ss.fund-domain}
Let $\Gamma$ be a Schottky subgroup of $\PGL(2,F)$.
Let us assume that the point at infinity~$\infty$ does not belong
to its limit set~$\mathscr L_\Gamma$.
Then, by~\cite[I, (4.3)]{gerritzen-vanderput1980},
the group~$\Gamma$ admits a basis $(\gamma_1,\dots,\gamma_g)$
and 
a \emph{good fundamental domain} $\mathfrak F_\Gamma$ with respect
to this basis, in the following sense:
\begin{enumerate}
\item There exists a finite family $(B_1,\dots,B_g,C_1,\dots,C_g)$ 
of strict bounded open disks in~$\P_1^\an$ such that 
$\mathfrak F_\Gamma=\P_1^\an\setminus \left(\bigcup B_i \cup \bigcup C_i\right)$;
\item The corresponding bounded 
closed disks $B_1^+,\dots,B_g^+,C_1^+,\dots,C_g^+$
are pairwise disjoint;

Let then $\mathfrak F_\Gamma^\circ=\P_1^\an\setminus \left(\bigcup B_i^+ \cup \bigcup C_i^+\right)$;
\item The elements $\gamma_1,\dots,\gamma_g$  satisfy
$\gamma_i( \P_1^\an\setminus B_i)=C_i^+$ and $\gamma_i(\P_1^\an\setminus B_i^+)=C_i$ for every $i\in\{1,\dots,g\}$.
\end{enumerate}
With this notation, let $W=\P_1^\an\setminus \bigcup B_i $;
this is an affinoid domain of $\P_1^\an$ containing~$\mathfrak F$,
stable under each~$\gamma_i$.
Indeed, one has $W\subset \P_1^\an\setminus B_i$, hence
$\gamma_i W\subset \gamma_i(\P_1^\an\setminus B_i)=C_i^+$,
hence the claim since $C_j^+$ is disjoint from each~$B_i$.

Moreover, the following properties are satisfied:
\begin{enumerate}\setcounter{enumi}{3}
\item One has $\bigcup_{\gamma\in\Gamma} \gamma\cdot \mathfrak F_\Gamma
=\P_1\setminus\mathscr L_\Gamma$;
\item
For $\gamma\in \Gamma$, one has 
$\mathfrak F_\Gamma\cap \gamma\cdot \mathfrak F_\Gamma\neq\emptyset$ if
and only if $\gamma\in\{\id,\gamma_1^{\pm1},\dots,\gamma_g^{\pm1}\}$;
\item For every $\gamma\in\Gamma\setminus\{\id\}$, one has 
$\mathfrak F_\Gamma^\circ\cap \gamma\cdot \mathfrak F_\Gamma=\emptyset$.
\end{enumerate}

In this context, we identify an element~$\gamma$
of~$\Gamma$ with a reduced word
in the letters 
$\{\gamma_1^{\pm},\dots,\gamma_g^{\pm}\}$ and
denote its length by~$\ell_\Gamma(\gamma)$.

For every~$\gamma\in\Gamma\setminus\{\id\}$, 
\cite[I, \S4, p.~29]{gerritzen-vanderput1980}
define a bounded open disk $B(\gamma)$,
equal either to $\gamma\cdot(\P_1^\an\setminus B_i^+)$
or to $\gamma\cdot(\P_1^\an\setminus C_i^+)$, according
to whether the last letter of the reduced word representing~$\gamma$
is~$\gamma_i$ or $\gamma_i^{-1}$; in any case, one has $\gamma\cdot\infty\in B(\gamma)$.
Moreover, they prove:
\begin{enumerate}\setcounter{enumi}{6}
\item $B(\gamma')\subset B(\gamma)$ if and only if 
$\gamma$ is an initial subword of~$\gamma'$;
\item For every integer~$n$, one has
\[ \P_1^\an\setminus \bigcup_{\ell_\Gamma(\gamma)<n} \gamma\cdot\mathfrak F
 = \bigcup_{\ell_\Gamma(\gamma)=n} B(\gamma); \]
\item  There exists a real number~$c>1$
such that for every~$\gamma$, the radius of the disk~$B(\gamma)$ 
is $\ll c^{-\ell_\Gamma(\gamma)}$; 
\item 
The intersection of every decreasing sequence of open disks $(B(\gamma_n))$,
where $\ell_\Gamma(\gamma_n)=n$,
is reduced to a limit point of~$\Gamma$,
and every limit point can be obtained in this way.
\end{enumerate}

\begin{prop}\label{prop.proper}
Let $\Gamma$ be a Schottky group in~$\PGL(2,F)$ and let 
$\mathfrak G$ be a compact analytic domain of~$\Omega_\Gamma$.
There exist positive real numbers~$a,b$ such that
for every~$\gamma\in\Gamma$ and every rigid point~$x\in\gamma\cdot\mathfrak G$,
one has
\[ \ell_\Gamma(\gamma)\leq a - b \log (\delta_\Gamma(x)). \]
\end{prop}
\begin{proof}
To prove this proposition, we may extend the scalars to a finite
extension of~$F$ and henceforth assume that the limit
set~$\mathscr L_\Gamma$ is not equal to~$\P_1(F)$.
Placing a point of $\P_1(F)\setminus\mathscr L_\Gamma$ at infinity,
\S\ref{ss.fund-domain} furnishes
a basis $(\gamma_1,\dots,\gamma_g)$ and
a good fundamental domain~$\mathfrak F$ with respect to this basis
of the form $\mathfrak F=\P_1^\an\setminus\left(\bigcup_{i=1}^g B_i\cup \bigcup_{i=1}^g C_i\right)$.  Let $b$ and $c>1$ be positive real numbers such that
the diameter of~$B(\gamma)$ is bounded by $b c^{-\ell_\Gamma(\gamma)}$,
for every $\gamma\in\Gamma\setminus\{\id\}$.

Let $x\in\Omega_\Gamma$ and let $\gamma\in\Gamma$ be such that $x\in\gamma\cdot\mathfrak F$.
Let $\xi\in\mathscr L_\Gamma(x)$ such that $\delta_\Gamma(x)=\delta(x,\xi)$.
Since the disk~$B(\gamma)$ contains both~$x$ and~$\xi$,
one has $\delta_\Gamma(x) \leq b c^{-\ell_\Gamma(\gamma)}$,
that is, 
\[ \ell_\Gamma(\gamma)
\leq \frac1{\log(c)} \left( - \log(\delta_\Gamma(x)) + \log(b) \right),
\]
since $\log(c)>0$. This proves the proposition in the particular
case where $\mathfrak G=\mathfrak F$. 

Let us now prove the general case. Let $a$ be a real number such that $\delta_\gamma(x)\geq a>0$ for every rigid point of~$\mathfrak G$ (lemma~\ref{lemm.minor}).
The preceding inequality shows that there exists a finite subset~$S$
of~$\Gamma$ such that $\mathfrak G$ meets~$\gamma\cdot\mathfrak F$
if and only if $\gamma\in S$.
It then follows from property~(8) that $\mathfrak G$ is contained 
in the finite union $\bigcup_{s\in S} s\cdot\mathfrak F$.
To conclude the proof, we observe that if $x\in\gamma\cdot\mathfrak G$,
then there exists~$s\in S$ such that $x\in \gamma s\cdot\mathfrak F$.
The proposition then follows from the 
particular case already treated and from the
inequality $\ell_\Gamma(\gamma ) \leq \ell_\Gamma(\gamma s)+\ell_\Gamma(s)$.
\end{proof}
\begin{coro}\label{coro.proper}
Let $\mathfrak G$ and~$\mathfrak G'$ be compact analytic domains of~$\Omega_\Gamma$.
The set of $\gamma\in\Gamma$ such that $\gamma\cdot\mathfrak G\cap\mathfrak G' \neq\emptyset$ is finite.
\end{coro}
\begin{proof}
Let $S$ be this set. For $\gamma\in S$, the intersection
$\gamma\cdot\mathfrak G\cap\mathfrak G'$ is a non-empty
affinoid domain of~$\P_1^\an$, hence it contains a rigid point~$x_\gamma$.
With $a$ and~$b$ as in the statement of proposition~\ref{prop.proper},  one has
$\ell_\Gamma(\gamma)\leq a - b \log(\delta_\Gamma(x_\gamma))$.
Since $x_\gamma\in\mathfrak G'$, $\delta_\Gamma(x_\gamma)$ is bounded
from below, by lemma~\ref{lemm.minor}. This shows that $\ell_\Gamma(\gamma)$ is bounded
above, when $\gamma$ runs over~$S$.
\end{proof}

\begin{prop}\label{prop.fundamental}
Let $\Gamma$ be a Schottky group in~$\PGL(2,F)$ and let~$g$ be its rank.
Let $\xi\in\mathscr L_\Gamma$ and let $U$ be an open neighborhood of~$\xi$
in~$\P_1^\an$.

There exist an open neighborhood~$U'$ of~$\xi$, contained in~$U$,
a basis $\gamma_1,\dots,\gamma_g$ of~$\Gamma$,
an affinoid domain~$\mathfrak F\subset\Omega_\Gamma$
such that the following properties hold:
\begin{enumerate}
\item One has $\mathfrak F\subset U'$;
\item For every $i$, one has $\gamma_i(U')\subset U'$;
\item One has $\bigcup_{\gamma\in\Gamma} \gamma\mathfrak F=\Omega_\Gamma$.
\end{enumerate}
\end{prop}
Such an affinoid domain will called a \emph{fundamental set.}
\begin{proof}
We first treat the case where $\mathscr L_\Gamma\neq\P_1(F)$.
Placing a point of $\P_1(F)\setminus\mathscr L_\Gamma$ at infinity,
\S\ref{ss.fund-domain} furnishes
a basis $(\gamma_1,\dots,\gamma_g)$ and
a good fundamental domain~$\mathfrak F$ with respect to this basis
of the form $\mathfrak F=\P_1^\an\setminus\left(\bigcup_{i=1}^g B_i\cup \bigcup_{i=1}^g C_i\right)$. 

By~(10), for every integer~$n\geq 1$, there is an element $\gamma \in\Gamma$ 
of length~$n$
such that $\xi\in B(\gamma)$; if $n$ is large enough,
one has $B(\gamma)^+\subset U$, because the diameter of~$B(\gamma)^+$
tends to~$0$ when $n=\ell_\Gamma(\gamma)$ tends to~$\infty$.
Since $\gamma\cdot\mathfrak F\subset B(\gamma)^+$, 
this implies that $\gamma\cdot\mathfrak F\subset U$.

Up to changing the basis $(\gamma_1,\dots,\gamma_g)$ into
$(\gamma_1^{-1},\dots,\gamma_g^{-1})$, and exchanging~$B_i$ and~$C_i$
for every~$i$, we may assume that
the last letter of~$\gamma$ is~$\gamma_s$, for some $s\in\{1,\dots,g\}$.
Set $W=\P_1^\an\setminus\bigcup_{i=1}^g B_i$;
recall that $W$ is an affinoid domain of~$\P_1^\an$ containing~$\mathfrak F$
and stable under~$\gamma_1,\dots,\gamma_g$.
By definition, one has
\[ B(\gamma)^+=\gamma \cdot (\P_1^\an\setminus B_s)\supset \gamma \cdot W, \]
since $W\subset\P_1^\an\setminus B_s$.

 
Let us now set $\mathfrak F'=\gamma\cdot\mathfrak F$, $W'=\gamma\cdot W$
and $\gamma'_i=\gamma\gamma_i\gamma^{-1}$ for $i\in\{1,\dots,g\}$.
By construction, $\mathfrak F'$ and~$W'$ are affinoid domains of~$\P_1^\an$
such that
$\mathfrak F'\subset W' \subset B(\gamma)^+\subset U$, 
the translates of~$\mathfrak F'$
under~$\Gamma$ cover~$\Omega_\Gamma$, and $W'$ is stable 
under the basis $(\gamma'_1,\dots,\gamma'_g)$ of~$\Gamma$.

This almost proves (1--3), except that~$W'$ is affinoid and not open.
To conclude the construction, one sets~$U'$ to be the interior of~$W'$,
and one redoes the construction starting from~$U'$ instead of~$U$.
The second paragraph of the proof shows that there exists~$\gamma'\in\Gamma$
such that $\gamma'\cdot\mathfrak F'$ is contained in~$U'$.
The affinoid affinoid~$\gamma'\cdot\mathfrak F'$, the open subset~$U'$
and the basis $(\gamma'_1,\dots,\gamma'_g)$ satisfy the 
requirements of the proposition.

Let us now treat the case where $\mathscr L_\Gamma=\P_1(F)$.
Let $F'$ be a finite extension of~$F$ of degree~$>1$.
The preceding construction can be applied starting with a point
of $\P_1(F')\setminus\mathscr L_\Gamma$ and furnishes
an open neighborhood~$V'$ of~$\xi$ in~$(\P_1^\an)_{F'}$, contained
in~$U_{F'}$,
a basis $(\gamma_1,\dots,\gamma_g)$ of~$\Gamma$
and an affinoid domain~$\mathfrak F'$ of~$\Omega_{\Gamma,F'}$ 
satisfying properties~\mbox{(1--3)}.
The images~$U'$ of~$V'$ 
and $\mathfrak F$ of~$\mathfrak F'$
by the projection $(\P_1^\an)_{F'}\to \P_1^\an$ satisfy the
required properties.
\end{proof}

\begin{lemm}\label{lemm.heights-length}
Let $\Gamma$ be an arithmetic Schottky group in~$\PGL(2,F)$
and let $H$ be a height function on $\PGL(2,\overline\Q)$.
There exists a positive real number~$c$
such that $H(\gamma)\leq c^{\ell_\Gamma(\gamma)+1}$, 
for every $\gamma\in\Gamma$.
\end{lemm}
\begin{proof}
Let $(\gamma_1,\dots,\gamma_g)$ be a basis of~$\Gamma$ as above.
Let $c_1$ be a positive real number such that $H( hh')\leq c_1 H(h)H(h')$
for every $h,h'\in \PGL(2,\overline\Q)$.
Let $c=c_1 \sup(H(\id),H(\gamma_1),\dots,H(\gamma_g))$.
One proves by induction on~$\ell_\Gamma(\gamma)$
that 
\[ c_1 H(\gamma) \leq \sup(c_1H(\gamma_1^{\pm}),\dots,c_1H(\gamma_g^{\pm}))^{\ell_\Gamma(\gamma)}  c_1 H(\id) \leq c_1  c^{\ell_\Gamma(\gamma)+1} \]
for every  $\gamma\in\Gamma$,
as was to be shown.
\end{proof}

\begin{lemm}\label{lemm.few-rational}
Let $\Gamma$ be a Schottky subgroup of $\PGL(2,F)$
and let $\Delta$ be a subset of $\P_1(\bar F)$
of cardinality~$2$.
Let $K$ be a number field contained in~$F$.
The stabilizer of~$\Delta$ inside~$\Gamma$
does not have many $K$-rational points.
\end{lemm}
\begin{proof}
Let $S$ be this stabilizer;
we may assume that $S\neq\{\id\}$.
Let $g\in S\setminus\{\id\}$. Then $g$ is hyperbolic
(see~\cite{gerritzen-vanderput1980}, page~7, line~2),
hence has exactly two rational fixed points in~$\P_1(F)$.
Up to a change of projective coordinates, we may thus
assume that $\Delta=\{0,\infty\}$.
Then every element~$h$ of~$S$ is of the form $z\mapsto \lambda(h) z$,
for some unique element $\lambda(h)\in K^\times$;
moreover, unless $h=\id$, any such~$h$ is hyperbolic
hence is represented by a matrix having two eigenvalues
with distinct absolute values, so that
$\abs{\lambda(h)}\neq 1$.
Let us choose $h\in S\setminus\{\id\}$
such that $\abs{\lambda(h)}$ is~$>1$ and minimal.
By euclidean division, one has $S=\langle h\rangle$.

Then $S\cap\PGL(2,K)$ is generated
by an element of the form~$h^a$, for some~$a\in\Z$.
Since $h^a$ is semisimple, we have 
$H(h^a)^n\ll H(h^{an})\ll H(h^a)^n$, for every  $n\in\Z$
(see~\S\ref{ss.heights}).
This shows that $S\cap \PGL(2,K)$ does not have many rational points.
\end{proof}

In \S\ref{sec.proof}, we will need the following lemma.
\begin{lemm}\label{lemm.distance}
Let $r$ be a positive real number 
and let $f\in\C_p\lbra z\rbra$ be a power series which
converges on the closed disk $E(0,r)$.
Let $L_1$ and~$L_2$ be closed subsets of~$\C_p$ 
such that $f^{-1}(L_2)\subset L_1$; for every $x\in\C_p$,
let $\delta(x;L_1)$ and~$\delta(x;{L_2})$ 
be the distances of~$x$ to~$L_1$ and~$L_2$ respectively.
Then there exists real numbers~$m\geq 0$, $c>0$, and~$s$ such that $0<s<r$
and such that $\delta(f(x);L_2)\geq c \delta(x;L_1)^m$
for every $x\in E(0,s)$.
\end{lemm}
\begin{proof}
Write $f=\sum c_n z^n$.
We may assume that there exists $a\in\C_p^\times$ 
such that $r=\abs a$; 
composing~$f$ with homographies which map $E(0,r)$ to $E(0,1)$
and $f(E(0,r))$ into the disk $E(0,1)$, 
we assume that $r=1$ and that $\abs {c_n}\leq 1$ for all~$n$.
(Recall from~\S\ref{ss.delta} that homographies are Lipschitz for the
distance~$\delta$.)

Let us first treat the case where $f(0)\not\in L_2$.
Then there exists a real number~$s>0$ such that 
$E(f(0),s)\cap L_2=\emptyset$.
For every $x\in E(0,1)$ such that $\abs x<s$, one has $\abs{f(x)-f(0)}<s$,
hence $\delta(f(x);L_2)>s$. It suffices to set $m=0$ and $c=s$.

We now assume that $f(0)\in L_2$, hence $0\in L_1$. Let $m=\ord_0(f-f(0))$.
Since $f'(z)=\sum_{n\geq m} nc_nz^{n-1}$, there exists a
real number~$s$ such that  $0<s\leq 1$ and such that
$\abs{f'(z)}=\abs{m c_m} \abs{z}^{m-1}$ provided $\abs z\leq s$.
Moreover, $\abs{f^{(n)}(z)/n!}\leq 1$ for every $n\geq0$
and any $z\in E(0,1)$.
Considering the Taylor expansion
\[ f(y)=\sum_{n\geq0} \frac1{n!}f^{(n)}(x) (y-x)^n, \]
we then see that there exists a real number~$s' $ such that 
\[ f( E(x,u)) = E(f(x), \abs{f'(x)} u) \]
for every real number~$u$ such that $0<u\leq s'$
and every $x\in E(0,1)$ such that $0<\abs x\leq s$.
If $u<\delta(x;L_1)$, then $E(x,u)\cap L_1=\emptyset$,
hence $E(f(x),\abs{f'(x)}u)\cap L_2=\emptyset$; consequently,
$\delta(f(x);{L_2})\geq \abs{f'(x)} \delta(x;{L_1})$.
Since $0\in L_1$, one has $\abs x\geq \delta(x;{L_1})$.
Consequently,
\[ \delta(f(x);L_2) \geq \abs{mc_m} \abs x^{m-1}\delta(x;{L_1})
 \geq \abs{mc_m} \delta(x;{L_1})^m. \]
This concludes the proof.
\end{proof}

\section{Automorphisms of curves}

The following result is already present in~\cite{pila2013}.
For the clarity of exposition, we isolate it as a lemma.

\begin{lemm}
Let $k$ be an algebraically closed field of characteristic zero,
let $B$ be a smooth connected projective $k$-curve,
let $f\colon B\to\mathbf P_1$ be a nonconstant morphism and
let $R_f\subset B$ be the ramification locus of~$f$
(the set of points of~$B$ at which $f$ is not étale)
and let $\Delta_f=f(R_f)$ be its discriminant locus.

Assume that there exist automorphisms~$g\in\Aut(\mathbf P_1)$ 
and $h\in\Aut(B)$ such that $f\circ h=g\circ f$.
Assume that~$g$ has infinite order.
Then $B$ is isomorphic to~$\P_1$, and one of the following cases holds:
\begin{itemize}
\item The morphism~$ f$ is an isomorphism (and $\Delta_f=\emptyset$);
\item One has $\Card(R_f)=2$, and $g(\Delta_f)=\Delta_f$.
\end{itemize}
\end{lemm}
\begin{proof}
By construction, $f$ induces a finite étale
covering of $\P_1\setminus\Delta_f$.

Let $b\in R_f$; one has $df(b)=0$, hence $d(f\circ h)(b)=d(g\circ f)(b)=0$.
Since $h$ is an automorphism of~$B$,
this implies that $df(h(b))=0$, hence $h(b)\in R_f$.
We thus have $h(R_f)\subset R_f$, hence $h(R_f)=R_f$, 
because $h$ is an isomorphism.
Consequently, $g(\Delta_f)=\Delta_f$,
so that some power of~$g$ fixes $\Delta_f$ pointwise.
Since the identity is the only homography that fixes $3$~points
and $g$ has infinite order, this implies that $\Card(\Delta_f)\leq 2$.

If $\Card(\Delta_f)\leq 1$, then $\P_1\setminus\Delta_f$ 
is simply connected, hence $f$ is an isomorphism (and $\Delta_f=\emptyset$).

Otherwise, one has $\Card(\Delta_f)=2$. Let $n=\deg(f)$.
Up to a change of projective coordinates in~$\mathbf P_1$,
we may assume that $\Delta_f=\{0,\infty\}$.
Then  $g$ is a homothety, because it leaves~$\Delta_f$ invariant
and has infinite order (otherwise, it would be of the form $g(z)=a/z$).
Since all finite étale coverings of $\P_1\setminus\Delta_f$ are of
Kummer type (equivalently, $\pi_1(\P_1\setminus\Delta_f)=\Z$),
one has $B\simeq\mathbf P_1$ and
the morphism~$f$ is conjugate to the morphism $z\mapsto z^n$
from~$\mathbf P_1$ to itself.

We then remark that $h$ is a homography of infinite order. Indeed,
if $h^e=\id_B$, then $f=g^{e}\circ f$, hence $g^e=\id$ since $f$ is surjective,
hence $e=0$, since $g$ has infinite order.
As above, the formula $h(R_f)=R_f$ then implies that $\Card(R_f)\leq 2$.
On the other hand, $\Card(R_f)\geq\Card(\Delta_f)=2$,
hence $\Card(R_f)=2$.
\end{proof}

\begin{prop}\label{prop.aut-curves}
Let $k$ be a field of characteristic zero.
Let $B$ be an integral $k$-curve in~$\P_1^n$ 
possessing a smooth $k$-rational point.
Let $\Gamma_B$ be the stabilizer of~$B$ in $(\Aut(\P_1))^n$
and let $\Gamma_1\subset \Aut(\P_1)$ be its image under the first projection.
Assume that $\Gamma_1$ contains an element of infinite order.
Then, one of the following cases holds:
\begin{enumerate}
\item The morphism $p_1|_B$ is constant;
\item The morphism $p_1|_B$ is an isomorphism
and the components of its inverse are either constant or homographies;
\item There is a subset of~$\P_1(\bar k)$ of cardinality~$2$ 
which is invariant under every element of~$\Gamma_1$.
\end{enumerate}
\end{prop}
\begin{proof}
Assume that $p_1|_B$ is not constant.
Let $\nu\colon B'\to B$ be the normalization of~$B$ and 
let $p'_1=p_1\circ\nu\colon B'\to\P_1$.
Let $g=(g_1,\dots,g_n)$ be an element of~$\Gamma_B$;
There exists a unique automorphism~$h$ of~$B'$ that lifts~$g$;
then $p'_1\circ h=g_1\circ p'_1$.
Since the curve~$B$ has smooth rational points, 
the curve~$B'$ is geometrically integral.
Choosing~$g$ such that $g_1$ has infinite order,
the preceding lemma implies that $\Card(R_{p'_1})\in\{0,2\}$.

Let us first assume that $\Card(R_{p'_1})=2$. Then
$\Card(\Delta_{p'_1})=2$ as well.
Moreover, the relation $p'_1\circ h=g_1\circ p'_1$ implies
that $g_1(\Delta_{p'_1})\subset \Delta_{p'_1}$, so that case~(3) holds.

Let us now assume that $\Card(R_{p'_1})=0$ and fix~$g$
such that $g_1$ has infinite order. By the preceding
lemma, $p'_1$ is an isomorphism;
this implies that $p_1|_B$ is an isomorphism as well.
Let $f$ be its inverse and 
let $f_1,\dots,f_n$ be its components; assume that 
case~(2) does not hold, that is, for some~$j$,
the rational map~$f_j$ is neither constant, nor a homography;
its ramification locus~$R_j$ is non-empty.
Since $g_1$ has infinite order, 
the relation $g_j\circ f_j=f_j\circ g_1$ implies that $g_j$
has infinite order as well.
By the preceding lemma, one has $\Card(R_j)=2$.
Let then $g'=(g'_1,\dots,g'_n)$ be any element of~$\Gamma_B$.
The relation $g'_j\circ f_j = f_j\circ g'_1$
implies that $g'_1(R_j)\subset R_j$,
so that case~(3) holds.
\end{proof}

\section{Proof of theorem~\ref{mt}}
\label{sec.proof}


We will reduce the proof of theorem~\ref{mt} to its following variant:

\begin{prop}\label{mt-bis}
Let $F$ be a finite extension of $\Q_p$ and
let $(\Gamma_i)_{1\leq i\leq n}$ be a finite family
of  \emph{arithmetic} Schottky subgroups of $\PGL(2,F)$ of ranks~$\geq 2$.
As above,
let us set $\Omega=\prod_{i=1}^n \Omega_{\Gamma_i}$ and
$X=\prod_{i=1}^n X_{\Gamma_i}$, 
and let $p\colon \Omega \to X^\an$ be the morphism deduced
from the morphisms $p_{\Gamma_i}\colon \Omega_{\Gamma_i}\to X_{\Gamma_i}^\an$.

Let $V$ be an irreducible algebraic subvariety of~$X$
and let $W$ be an irreducible algebraic
subvariety of~$\Omega$, maximal among those contained in $p^{-1}(V^\an)$. 
If $W$ is geometrically irreducible, then it is flat.
\end{prop}

\begin{lemm}
Proposition~\ref{mt-bis} implies theorem~\ref{mt}.
\end{lemm}
\begin{proof}
Let $Y$ be the Zariski closure of~$W$ in~$\P_1^n$; by assumption,
$W$ is an irreducible component of~$Y^\an\cap\Omega$.
Let $W_0$ be an irreducible component of~$W_{\C_p}$.
By~\cite[théorème~7.16, (v)]{ducros2009}, there exists
a finite extension~$F'$ of~$F$, contained in~$\C_p$,
and an irreducible component~$W'$
of~$W_{F'}$ such that $W_0=W'_{\C_p}$.
Then $W'$ is geometrically irreducible, as well as its
Zariski closure~$Y'$. By proposition~\ref{lemm.irred},
$\Omega\cap Y'$  is geometrically
irreducible;
the inclusion $W'\subset \Omega\cap Y'$
and the inequality $\dim(W')=\dim (W_0)=\dim(W)=\dim(Y)\geq\dim(Y')$
imply that $W'=\Omega\cap Y'$; in particular, $W'$ is irreducible
algebraic and is contained in $p^{-1}(V_{F'}^\an)$. 
Let us show that it is maximal.
Let $W'_1\subset \Omega_{F'}$ be an irreducible algebraic subvariety
contained in $p^{-1}(V_{F'}^\an)$ such that $W'\subsetneq W'_1$;
let $Y'_1\subset(\P_1^n)_{F'}$ be the Zariski closure of~$W'_1$. 
The image~$Y_1$ of~$Y'_1$ in~$(\P_1^n)_F$ is Zariski closed, 
because $F'$ is a finite extension of~$F$, 
and $Y'_1\subset (Y_1)_{F'}$. Moreover, $Y\subset Y_1$.
There exists a unique irreducible component~$W_1$ of~$\Omega\cap Y_1$
that contains~$W$, and $W'_1$ is an irreducible component
of~$W_{1,F'}$. 
Necessarily, $W_1$ is contained in $p^{-1}(V^\an)$, because
$W'_1\subset p^{-1}(V_{F'}^\an)$; this contradicts the maximality
of~$W$. 

Applying proposition~\ref{mt-bis} to~$W'$, we conclude that
$W'$ is flat. Consequently, $W_0=W'_{\C_p}$ is flat, as
was to be shown.
\end{proof}

\subsection{}
To prove proposition~\ref{mt-bis},
we argue by induction and assume that it
holds if there are less that $n$ factors.
Let $W$ be an irreducible algebraic subvariety of~$\Omega$,
maximal among those 
contained in $p^{-1}(V^\an)$ and geometrically irreducible.
Let $Y$ be an irreducible subvariety of~$\P_1^n$ such that
$W$ is an irreducible component of $Y^\an\cap\Omega$.
By corollary~\ref{coro.irr-comp-2}, $Y$ is geometrically irreducible,
$W=Y^\an\cap\Omega$ and $W$ is topologically dense in~$Y$. 

 
The proof that $W$ is flat requires intermediate steps and will be concluded
in proposition~\ref{prop.W-flat}.

A crucial step will consist in proving that 
the stabilizer of~$W$ inside $\Gamma$ 
has many points of bounded heights 
(proposition~\ref{prop.stabilizer-infinite}). 
To that aim, we define in section~\ref{sec.R} 
an $F$-subanalytic subset~$R$ of
$\PGL(2,F)^n$. The definition, close to that
of a similar set in~\cite{pila:2011,pila2015},
guarantees the following important property
(lemma~\ref{lemm.R}): if  $B$ is a small enough subset of~$R$ then,
for every $g\in B$, 
the translate $(g\cdot Y^\an)\cap\Omega$ is contained in $p^{-1}(V^\an)$, 
and is independent of~$g$. At this point, the maximality of~$W$
is invoked.

The existence of such blocks  is established by applying the
$p$-adic Pila-Wilkie theorem of~\cite{cluckers-comte-loeser2015}.
We thus prove that $R$ has many rational points (lemma~\ref{lemm.R-many});
these points are constructed using the action of the Schottky groups
in a neighborhood of a boundary point~$\xi$,
applying material recalled in section~\ref{sec.schottky}.
The construction of such a point~$\xi$, performed
in lemma~\ref{lemm.projection}, is actually the starting point
of the proof.

The actual statement of proposition~\ref{prop.stabilizer-infinite}
furnishes elements in~$\Gamma$ of a precise form. 
Using proposition~\ref{prop.aut-curves},
we will finally conclude the proof of proposition~\ref{mt-bis}.

\subsection{}\label{ss.section1}
By assumption, $W=Y^\an\cap\Omega$;
consequently, the $j$th projection
$q_j\colon(\P_1)^n\to \P_1$ is constant on~$Y$ if and only if
it is constant on~$W$, if and only if
the $j$th projection from~$X$ to~$X_j$ is constant  on~$V$,
and in this case, its image is an $F$-rational point of~$\P_1$,
because $W$ is geometrically irreducible.
Deleting these constant factors,
we thus assume that there does not exist $j\in\{1,\dots,n\}$
such that the $j$th projection $q_j\colon(\P_1)^n\to \P_1$ is constant on~$Y$.
Consequently, $q_j|_Y\colon Y\to\P_1$ is surjective, for every~$j$;
in particular, $Y^\an $ meets $q_j^{-1}(\mathscr L_{\Gamma_j})$.

Let $m=\dim(Y)$; by what precedes, we have $m>0$,
and $Y^\an\not\subset\Omega$.

\begin{lemm}\label{lemm.projection}
Up to reordering the coordinates, 
there exists
a smooth rigid point $\xi\in Y^\an$
and a connected open neighborhood~$U$ of~$\xi$ in~$(\P_1^n)^\an$
such that the following properties hold:
\begin{enumerate}
\item
The first component~$q_1(\xi)$ of~$\xi$ belongs to
the limit set~$\mathscr L_{\Gamma_1}$ of~$\Gamma_1$;
\item
Letting $J=\{1,\dots,m\}$,
the projection $q_J\colon \P_1^n\to \P_1^J$ induces a finite
étale morphism 
from~$U\cap Y^\an$ to its image in~$(\P_1^J)^\an$;
\item
For every $j\in\{1,\dots,n\}$ and every point $y\in U\cap Y^\an$
such that $q_j(y)\in \mathscr L_{\Gamma_j}$,
one has $q_1(y)\in\mathscr L_{\Gamma_1}$.
\end{enumerate}
\end{lemm}
\begin{proof}
For every subset~$V$ of~$Y^\an$,
let us define a relation~$\preceq_V$ on~$\{1,\dots,n\}$
as follows: $i\preceq_V j$ if and only if, for every $y\in V$
such that
$q_i(y)\in\mathscr L_{\Gamma_i}$, one has $q_j(y)\in\mathscr L_{\Gamma_j}$.
This is a preordering relation. 
If $U\subset V\subset Y^\an$ and $i\preceq_V j$, then
$i\preceq_U j$.

We will define a decreasing sequence $(V_0,V_1,\dots,V_n)$
of non-empty open subsets of~$Y^\an$
and a sequence $(j_0,j_1,\ldots, j_n)$ 
of elements of~$\{1,\dots,n\}$,
such that 
for every~$k$,
$q_{j_k}(V_k)$ meets~$\mathscr L_{\Gamma_{j_k}}$ and
$1,\dots,k \preceq_{V_k} j_k$.

We start with $V_0=Y^\an$. We have reduced ourselves to the
case where  $q_j(Y^\an)=\P_1$ for every~$j$; in particular,
$q_j(Y^\an)$ meets $\mathscr L_{\Gamma_{j}}$; we may take $j_0=1$.

Let $k\geq0$ be such 
that $V_0,V_1,\dots,V_k$ and $j_0,j_1,\dots,j_{k}$ are defined.
If $k+1\preceq_{V_k} j_k$, we set $V_{k+1}=V_k$ and $j_{k+1}=j_k$.
Otherwise, one has $k+1\not\preceq_{V_k} j_k$,
hence there exists $y\in V_k$ such that
$q_{k+1}(y)\in\mathscr L_{\Gamma_{k+1}}$
and $q_{j_k}(y)\not\in\mathscr L_{\Gamma_{j_k}}$.
Let $V_{k+1}=V_k \cap (q_{j_k})^{-1}( \Omega_{\Gamma_{j_k}})$;
this is an open neighborhood of~$y$ in~$V_k$
such that $q_{j_{k+1}}(V_{k+1})$ meets $\mathscr L_{\Gamma_{j_{k+1}}}$.
By construction, no element~$z$ of~$V_{k+1}$
satisfies $q_{j_k}(z)\in\mathscr L_{\Gamma_{j_k}}$,
so that $j_k \preceq_{V_{k+1}} k+1$.
We then set $j_{k+1}=k+1$.

Let $V=V_n$ and $i=j_n$;
let $y\in V$ such that $q_i(y)\in\mathscr L_{\Gamma_i}$.
Let $Z$ be the dense open subscheme of~$Y$ consisting of smooth points
at which $dq_i$ does not vanish.
Then $Z^\an$ is open and dense in~$Y^\an$, 
and $V\cap Z^\an$ is open and dense in~$V$,
hence $q_i(V\cap Z^\an)$ is dense in~$q_i(V)$.
Since $\mathscr L_{\Gamma_i}$ has no isolated points,
we may assume that $y\in Z^\an$.
Rigid points are dense in~$q_i^{-1}(q_i(y))\cap V\cap Z^\an$;
there exists a rigid point~$\xi$ in $(q_i)^{-1}(q_i(y))\cap V\cap Z^\an$.
Since $q_i(y)$ is a rigid point, the point~$\xi$ is a rigid point
of~$V\cap Z^\an$ (and not only of its fiber of~$q_i$).
Moreover, $q_i(\xi)=q_i(y)\in\mathscr L_{\Gamma_i}$.

Since $dq_i$ does not vanish at~$\xi$, 
there exists a subset~$J$ 
of~$\{1,\dots,n\}$ containing~$i$ such that
the projection~$q_J$ from~$V$ to~$(\P_1^J)^\an$
is finite étale  at~$\xi$. One has $\Card(J)=\dim(V)=m$.
Consequently, there exists an open neighborhood~$U$ of~$\xi$ in~$(\P_1^n)^\an$
such that $q_J$ induces a finite étale morphism from~$U\cap Y^\an$
to its image in $(\P_1^J)^\an$.

Reordering the coordinates, we may assume that $i=1$ and $J=\{1,\dots,m\}$,
hence the lemma.
\end{proof}

\subsection{}\label{ss.section2}
Choose~$\xi$, $J=\{1,\dots,m\}$ and~$U$ as in the previous lemma; 
we may even assume
that $U$ is of the form $U_1\times\dots \times U_n$,
where, for each~$i$, $U_i$ is an open neighborhood of~$q_i(\xi)$
in~$\P_1^\an$. 

Let $F'$ be a finite extension of~$F$ such that $\xi\in Y(F')$.
Since $W$ is geometrically irreducible, 
$W_{F'}$ is an irreducible algebraic subvariety of~$\Omega$.
It is also maximal. 
Note that the flatness of~$W_{F'}$ implies the flatness of~$W$.
Replacing~$F$ by~$F'$, we thus may assume that $\xi\in Y(F)$;
then $q_J$ induces a local isomorphism at~$\xi$.

Let $\phi=(\phi_1,\dots,\phi_n)\colon  O \to Y^\an\cap U$ 
be an analytic section of~$q_J|_{Y^\an\cap U}$, 
defined on an open neighborhood~$O$ of~$q_J(\xi)$;
we may assume that $O=U_1\times\dots\times U_m$.

By condition~(3) of lemma~\ref{lemm.projection}
one has 
$q_1(\phi_j^{-1}(\mathscr L_{\Gamma_j}))\subset \mathscr L_{\Gamma_1}$
for every~$j\in\{1,\dots,n\}$.

\subsection{}\label{sec.R}
Let $G$ be the $\Q$-algebraic group $\PGL(2)^n$,
and 
let $G_0$ be the algebraic subgroup of~$G$ 
defined by
\begin{equation}
 (g_1,\dots,g_n)\in G_0 \quad\Leftrightarrow\quad g_2=\dots=g_m=1.
\end{equation}
We denote by $q_1,\dots,q_n$ the projections of~$G$ to~$\PGL(2)$.
For every compact analytic domain~$\mathfrak F$ of~$\Omega$,
we define a subset $R_{\mathfrak F}$ of~$G_0(F)$ by 
\begin{equation}
g\in R_{\mathfrak F} \quad\Leftrightarrow\quad
   \dim(g \cdot Y^\an\cap \mathfrak F \cap p^{-1}(V^\an))=m .
\end{equation}

\begin{lemm}\label{lemm.R}
Let $\mathfrak F$ be an affinoid domain of~$\Omega$.
\begin{enumerate}
\item
The set $R_{\mathfrak F}$ is an $F$-subanalytic subset of $G_0(F)$.
\item
For every $g\in R_{\mathfrak F}$, 
one has $(g\cdot Y^\an) \cap\Omega\subset p^{-1}(V^\an)$.
\item
Let $M\subset R_{\mathfrak F}$ be a subset whose Zariski closure is irreducible;
for every $g,h\in M$, one has $g\cdot Y=h\cdot Y$.
\end{enumerate}
\end{lemm}
\begin{proof}
\begin{enumerate}
\item
The sets~$V$ and~$Y$ are algebraic over~$F$, hence 
$V(\C_p)$ and $Y(\C_p)$ are rigid $F$-subanalytic.
Since $\mathfrak F$ is affinoid, the morphism $p|_{\mathfrak F}$
defines a rigid $F$-subanalytic map from~$\mathfrak F(\C_p)$
to $V(\C_p)$, 
so that $(\mathfrak F\cap p^{-1}(V^\an))(\C_p)$ 
is a rigid $F$-subanalytic set.
Consequently, 
taking $\C_p$-points,
$(g\cdot Y^\an\cap\mathfrak F\cap p^{-1}(V^\an))_g$ furnishes
a rigid $F$-subanalytic
family of rigid $F$-subanalytic subsets of~$\Omega(\C_p)$, 
parameterized by~$G_0(\C_p)$.
By b-minimality, the set of points $g\in G_0(\C_p)$
such that 
   $\dim(g\cdot Y^\an\cap \mathfrak F \cap p^{-1}(V^\an))=m $ is a rigid
$F$-subanalytic subset of $G_0(\C_p)$.
It then follows from lemma~\ref{lemm.remark}
that $R_{\mathfrak F}$ is an $F$-subanalytic subset of $G_0(F)$.
\item
Let $g\in R_{\mathfrak F}$
and let us prove that $(g\cdot Y^\an)\cap\Omega \subset p^{-1}(V^\an)$.
Since $g\cdot Y^\an$ is irreducible
and $g\cdot Y^\an\cap \mathfrak F$ has dimension~$m=\dim(g\cdot Y^\an)$,
this intersection is Zariski dense in $g\cdot Y^\an$.
Moreover, there exists a finite extension $F'$ of~$F$ 
such that $g\cdot Y^\an_{F'}\cap \mathfrak F(F')$ is Zariski dense in~$Y_{F'}$
(it suffices that $g\cdot Y^\an\cap \mathfrak F$ admits a smooth $F'$-point),
so that the Zariski closure of $g\cdot Y^\an\cap \mathfrak F(F')$ 
in~$(\P_1^n)_{F'}$ 
is equal to~$g\cdot Y_{F'}$.
Moreover, $g\cdot Y(F')\cap \mathfrak F(F')$ is $F'$-semialgebraic, 
hence proposition~\ref{prop.closure-analytic}
implies that 
$g\cdot Y_{F'}^\an\cap\Omega_{F'}\subset p_{F'}^{-1}(V_{F'}^\an)$.
Since $p$ is defined over~$F$ and $g\in G(F)$, 
this implies that $(g\cdot Y^\an)\cap\Omega\subset p^{-1}(V^\an)$. 
\item
As a subset, $(M\cdot Y^\an)\cap\Omega$ is contained in $p^{-1}(V^\an)$.
By proposition~\ref{prop.closure-analytic},
its Zariski closure~$Y'$ satisfies
$(Y')^\an\cap\Omega \subset  p^{-1}(V^\an)$ as well.
Since $Y$ and the Zariski closure of~$M$ are
geometrically irreducible, $Y'$ is geometrically irreducible.

Let $g\in M$; then $Y^\an\subset g^{-1}M\cdot Y^\an\subset g^{-1}\cdot (Y')^\an$, hence $W\subset g^{-1}\cdot (Y')^\an\cap\Omega$.
By maximality of~$W$, one has $W=g^{-1}\cdot (Y')^\an\cap\Omega$.
This implies $g\cdot Y=Y'$, thus $g\cdot Y=h\cdot Y$ for every $g,h\in M$.
\qedhere
\end{enumerate}
\end{proof}

We return to the context of \S\ref{ss.section2}. In particular,
$\xi$ is a point of $Y(F)$ such that $q_1(\xi)\in\mathscr L_{\Gamma_1}$,
and the restriction to~$Y$ of the projection to the first~$m$ coordinates 
is étale at~$\xi$, with a local analytic section~$\phi$ defined
on $U_1\times\dots\times U_m$.

\begin{lemm}\label{lemm.R-many}
There exist a real number $c>0$,
fundamental sets $\mathfrak F_i\subset \Omega_{\Gamma_i}$
and a subset $\Upsilon$ of $R_{\mathfrak F}\cap\Gamma$,
where $\mathfrak F=\prod\mathfrak F_i$, such that
\begin{enumerate}
\item For all $T$ large enough, one has $\Card(\Upsilon_T)\geq T^c$,
where $\Upsilon_T$ denotes the set of all $\gamma\in\Upsilon$
such that $H(\gamma)\leq T$;
\item The projection $q_1$ is injective on $\Upsilon$;
\item For all~$j\in\{1,\dots,n\}$
such that $q_j(\xi)\not\in\mathscr L_{\Gamma_j}$,
one has $\Card(q_j(\Upsilon))= 1$.
\end{enumerate}
\end{lemm}
We recall that there exists a number
field~$K$ contained in~$F$ such that $\Gamma\subset\PGL(2,K)^n$,
and $H$ is induced by a fixed height function on $\PGL(2,\overline\Q)^n$.
In particular, lemma~\ref{lemm.R-many} implies
that the subset~$R_{\mathfrak F}$ of~$\PGL(2,F)^n$
has many $K$-rational points, in the sense of section~\ref{subsec.many}.
\begin{proof}
Let $q$ be the genus of~$X_{\Gamma_1}$;
by proposition~\ref{prop.fundamental}, 
there exists a basis $\alpha_1,\dots,\alpha_q$ of~$\Gamma_1$, 
an open neighborhood~$U'_1$ of~$q_1(\xi)$, 
which is contained in~$U_1$ 
and stable under the action of~$\alpha_1,\dots,\alpha_q$,
and a fundamental set~$\mathfrak F_1$ for~$\Gamma_1$
contained in~$U'_1$. 
For simplicity of notation, we now assume that $U_1=U'_1$.

We have introduced in \S\ref{ss.section2}
a local analytic section 
$\phi=(\phi_1,\dots,\phi_n)\colon U_1\times\cdots\times U_m
\to Y^\an \cap U_1\times\dots\times U_n$ of the projection
$q_J\colon Y \to \P_1^J$, where $J=\{1,\dots,m\}$.
Let $j\in\{1,\dots,n\}$ be such that 
$q_j(\xi)\not\in\mathscr L_{\Gamma_j}$.
Then $q_j(\xi)$
has a compact analytic neighborhood~$U'_j$ contained in~$\Omega_{\Gamma_j}$.
Shrinking $U_1,\dots,U_m$ if necessary, 
we assume that the image of~$\phi_j$ is contained in~$U'_j$
for every such~$j$.

Let $a'=(a_1,\dots,a_n)\in W$ 
be a rigid point that belongs to the image of~$\phi$ and such that $a_1\in\mathfrak F_1$.
Let $a=(a_1,\dots,a_m)$; we have $a'=\phi(a)$.
For $j\in\{2,\dots,n\}$, we also choose a fundamental set~$\mathfrak F_j$
that contains~$a_j$.

We claim that we can complete any element~$\gamma_1\in F_1$ 
which is a positive word~$\gamma_1$ in~$\alpha_1,\dots,\alpha_q$
to an element~$\gamma\in \Gamma$
such that $\gamma^{-1}\in R_{\mathfrak F}$ and
$H(\gamma)\ll c^{\ell_{\Gamma_1}(\gamma_1)}$, for some real number~$c$.

Let us now prove the asserted claim.
For any positive word~$\gamma_1$ in~$\alpha_1,\dots,\alpha_q$,
one has $\gamma_1\cdot a_1\in U_1$;
in particular, we can consider
the point $a(\gamma_1)=(\gamma_1\cdot a_1,a_2,\dots,a_m)$
of $ U_1\times\dots\times U_m$
and its image $\phi(a(\gamma_1))$ under the section~$\phi$.

By \S\ref{ss.lipschitz},
there exists a real number~$c_1\geq 1$ such that 
$\delta(\alpha_j\cdot a_1;\mathscr L_{\Gamma_1}) \geq c_1^{-1} 
\delta(a_1;\mathscr L_{\Gamma_1})$,
uniformly in~$a_1$. By induction on the length~$\ell_{\Gamma_1}(\gamma_1)$
of the positive
word~$\gamma_1$, this implies the inequality
\begin{equation}\label{eq.lipschitz-length}
 \delta(\gamma_1 \cdot a_1;\mathscr L_{\Gamma_1}) \geq c_1^{-\ell_{\Gamma_1}(\gamma_1)}.
\end{equation}

We first set $\gamma_2=\dots=\gamma_m=1$

Let then $j>m$.
Let $\psi_j\colon U_1\to U_j$ be the analytic map
defined by $\psi_j(x)=\phi_j(x,a_2,\dots,a_m)$.
By construction (lemma~\ref{lemm.projection}), 
if $\psi_j(x)=\phi_j(x,a_2,\dots,a_m) \in\mathscr L_{\Gamma_j}$, 
one has $x=q_1(x,a_2,\dots,a_m)\in\mathscr L_{\Gamma_1}$.
In other words, one has $\psi_j^{-1}(\mathscr L_{\Gamma_j})\subset\mathscr L_{\Gamma_1}$.
Applying lemma~\ref{lemm.distance} to~$\psi_j$,
we obtain an inequality of the form
\[ \delta(\phi_j(x,a_2,\dots,a_m);\mathscr L_{\Gamma_j})\gg
   \delta(x;\mathscr L_{\Gamma_1})^k,\]
for some integer~$k\geq0$,
for all $x\in U_1$.
In particular,
\begin{equation}\label{eq.distance}
\delta(\phi_j(a(\gamma_1));\mathscr L_{\Gamma_j})\gg
   \delta(\gamma_1\cdot a_1;\mathscr L_{\Gamma_1})^k.
\end{equation}

By proposition~\ref{prop.fundamental}, 
there exists $\gamma_j\in\Gamma_j$ such that
$\phi_j(a(\gamma_1))\in\gamma_j\cdot \mathfrak F_j$.
By proposition~\ref{prop.proper} and lemma~\ref{lemm.heights-length},
one has 
\begin{equation}\label{eq.height}
 H(\gamma_j) \ll \delta(\phi_j(a(\gamma_1) );\mathscr L_{\Gamma_j})^{-\kappa}, \end{equation}
where $\kappa$ is a positive real number, independent of~$\gamma_1$.
By equations~\eqref{eq.lipschitz-length}, \eqref{eq.distance}
and~\eqref{eq.height}, we thus have
\begin{equation}
 H(\gamma_j)\ll \delta(\gamma_1\cdot a_1;\mathscr L_{\Gamma_j})^{-k\kappa}
 \ll c_1^{\ell_{\Gamma_1}(\gamma_1) k \kappa}.
\end{equation}
Let $c=c_1^{k\kappa}$.

Let $\gamma=(\gamma_1,\dots,\gamma_n)\in\Gamma$.
By what precedes, $H(\gamma)\ll c^{\ell_{\Gamma_1}(\gamma_1)}$.
Moreover, $\phi_j(a(\gamma_1))\in\gamma_j\cdot\mathfrak F_j$
for every~$j$: this follows from the fact that $a_j\in\mathfrak F_j$
if $j\leq m$, and from the construction of~$\gamma_j$ if~$j>m$.

Let us prove $\gamma^{-1} \in R_{\mathfrak F}$.
One has $W\subset p^{-1}(V^\an)$ by assumption; 
since $\gamma\in\Gamma$,
this implies $\gamma^{-1}\cdot W\subset p^{-1}(V^\an)$.
Consequently,
\[ \gamma^{-1}\cdot  Y^\an \cap\mathfrak F\cap p^{-1}(V^\an) 
\supset \gamma^{-1}\cdot  W \cap\mathfrak F\cap p^{-1}(V^\an) 
= \gamma^{-1} \cdot W\cap\mathfrak F. \]
The analytic morphism 
\[ U_1\times\dots U_m\to W , \qquad
  (x_1,\dots,x_m)\mapsto  \phi (\gamma_1\cdot x_1,x_2,\dots,x_m) \]
is an immersion and
maps the point~$a=(a_1,\dots,a_m)$ to the point $\phi(a(\gamma_1))\in\gamma\cdot\mathfrak F$. Since $a$ is  a rigid point,
this morphism maps a neighborhood of~$a$ into~$\gamma\cdot\mathfrak F$,
so that $\dim(W\cap \gamma\cdot\mathfrak F)\geq m$.
This proves $\gamma^{-1}\in R_{\mathfrak F}$.

Applying lemma~\ref{lemm.heights-length}  to estimate $H(\gamma_1)$,
we thus have shown the existence of a positive real number~$c$
such that for every positive word~$\gamma_1$ in $\alpha_1,\dots,\alpha_q$,
there exists an element $\gamma=(\gamma_1,\dots,\gamma_n)$
completing~$\gamma_1$
such that $H(\gamma)\ll c^{\ell_{\Gamma_1}(\gamma_1)}$
and $\gamma ^{-1}\in R_{\mathfrak F} \cap \Gamma$.

Let $\Upsilon'$ be the set of all such elements~$\gamma^{-1}$,
where $\gamma_1$ ranges over positive words in $\alpha_1,\dots,\alpha_q$.
It is a subset of $R_{\mathfrak F}\cap \Gamma$.
By construction, the projection~$q_1$ is injective on~$\Upsilon'$.
Moreover, since the number of positive words 
of length~$\ell$  in~$\alpha_1,\dots,\alpha_q$
is $q^\ell$, 
the cardinality of $\Upsilon'_T$ is bounded from below by
$q^{\log(T)/\log(c)}=T^{\log(q)/\log(c)}$,
and the exponent of~$T$ is strictly positive, since $q\geq 2$. 
Finally, let~$j$ be such that
$q_j(\xi)\not\in\mathscr L_{\Gamma_j}$.
By construction, $\phi_j(a(\gamma_1))\in\gamma_j\mathfrak F_j$,
hence $\gamma_j\mathfrak F_j$ meets~$U'_j$.
By corollary~\ref{coro.proper}, the set~$S_j$
of such  elements~$\gamma_j$ in~$\Gamma_j$ is finite.
It follows that there is a subset~$\Upsilon$ of~$\Upsilon'$
that satisfies the conclusion of the proposition.
\end{proof}

\begin{prop}\label{prop.stabilizer-infinite}
Let $G'_0$ be the subgroup of~$G_0$ consisting of elements
$(g_j)$ such that $g_j=\id$ if $q_j(\xi)\not\in\mathscr L_{\Gamma_j}$.
Both the stabilizer of~$W$ inside~$G'_0\cap \Gamma$
and its image in~$\Gamma_1$ under the first projection
have many rational points.
\end{prop}
\begin{proof}
Let $c,\Upsilon,\mathfrak F_i,\mathfrak F=\prod\mathfrak F_i,R=R_{\mathfrak F}$ be as given by lemma~\ref{lemm.R-many}; let $T_0>1$
be such that $\Card(\Upsilon_T)\geq T^c$ for $T\geq T_0$.

Let $K$ be a number field contained in~$F$ such that all groups~$\Gamma_j$
are contained in~$\PGL(2,K)$;
the points of~$R\cap \Gamma$ are $K$-rational points.
Recall that  for every real number~$T$, 
we denote by $R(K;T)$ the set of $K$-rational points of~$R$
of height~$\leq T$. One has $\Upsilon_T=\Upsilon\cap R(K;T)$.

Since $R$ is $F$-subanalytic (lemma~\ref{lemm.R}),
it is also $\Q_p$-subanalytic and
we may apply the $p$-adic Pila-Wilkie theorem
of~\cite{cluckers-comte-loeser2015}, as stated in theorem~\ref{theo.pw}.
Let thus $s\in\N$, $d\in\R$, $\eps>0$,
 and $B\subset  R\times\Q_p^s$ be a family of blocks
such that for every $T>1$, there exists a subset $\Sigma_T\in\Q_p^s$
of cardinality $<d T^\eps$ such that 
$R(K;T)\subset\bigcup_{\sigma\in \Sigma_T}B_\sigma$.
Let also $t\in\N$ and $Z\subset G_0(F)\times \Q_p^t$ be a semi-algebraic
subset such that for every $\sigma\in\Q_p^s$, there exists $\tau\in\Q_p^t$
such that $B_\sigma\subset Z_\tau$ and $\dim(B_\sigma)=\dim(Z_\tau)$.
Let finally~$r$ be an upper bound 
for the number of irreducible components 
of the Zariski closure of the sets~$Z_\tau$, for $\tau\in\Q_p^t$.

Let $T>T_0$.
Since $\Upsilon_T\subset R(K;T)$,
by the pigeonhole principle, there exists $\sigma\in\Sigma_T$
such that 
\[ \Card(\Upsilon_T \cap B_\sigma) \geq
    \frac{\Card(\Upsilon_T)}{\Card(\Sigma_T)} \geq \dfrac1d T^{c-\eps}. \]
Moreover, the Zariski closure of~$B_\sigma$ in $\PGL(2)_F^n$ has at most 
$r$~irreducible components. Consequently, we may choose
such an irreducible component~$\overline M$
whose trace~$M$ on~$B_\sigma$  satisfies
\[ \Card ( \Upsilon_T\cap M) \geq \dfrac1{dr} T^{c-\eps}. \]
(Observe that $\overline M$ is  indeed the Zariski closure of~$M$.)

Let $g\in \Upsilon_T\cap M$. 
Since the Zariski closure of $M$
is irreducible
and $M\subset R_{\mathfrak F}$,
it follows from lemma~\ref{lemm.R}
that the stabilizer of~$W$ inside~$G_0\cap \Gamma$
contains $g^{-1} M$, hence $g^{-1}(\Upsilon_T\cap M)$.
By construction, the image of $g^{-1}(\Upsilon_T\cap M)$
under the projection of index~$j$ is $\{\id\}$ 
if $q_j(\xi)\not\in\mathscr L_{\Gamma_j}$.
This shows in particular that the stabilizer of~$W$ 
inside~$G'_0\cap \Gamma$ contains $g^{-1}(\Upsilon_T\cap M)$.
This set contains $\geq T^{c-\eps}/dr$ points, and their heights
are $\ll T^2$; the same holds for its image by the first projection,
since this projection is injective 
on~$g^{-1}(\Upsilon\cap M)$.

We thus have shown that the stabilizer of~$W$ 
inside~$G'_0\cap \Gamma$ has many rational points,
as well as its image under the first projection.
This concludes the proof of the proposition.
\end{proof}

\begin{prop}\label{prop.W-flat}
The subvariety~$W$ is flat.
\end{prop}
\begin{proof}
We have constructed in~\S\ref{ss.section2} an analytic map
$\phi\colon U_1\times\dots\times U_m \to Y$, 
which is a local section of the
projection to the $m$~first coordinates.

Let $a\in \prod_{i=2}^m \left(\Omega_{\Gamma_i}\cap U_i\right)$;
let us denote by $W_a$ the fiber of~$W$ over~$a$
under the projection to $\prod_{i=2}^m \P_1^\an$;
let us define~$Y_a$ similarly. When $a$ varies,
the number of irreducible components of~$Y_a$
is uniformly bounded.

Let $\psi_a\colon (U_1)_{\mathscr H(a)} \to Y_a^\an$ be the analytic
morphism deduced from~$\phi$. 
We claim that the components of~$\psi_a$  are
either constant or homographies.

Let $g\in G_0\cap\Gamma$ be an element such that
$g\cdot W=W$, $g_1\neq \id$,
and $g_j=\id$ if $q_j(\xi)\not\in\mathscr L_{\Gamma_j}$
(proposition~\ref{prop.stabilizer-infinite}).
Since $g\cdot W=W$, one has $g\cdot Y=Y$, 
hence $g\cdot W_a=W_a$ and $g\cdot Y_a=Y_a$.
The element~$g$ induces a commutative diagram
\[ \begin{tikzcd} 
   Y_a \ar{r}{g} \ar{d}{} & Y_a \ar{d}{}  \\
    (\P_1)_{\mathscr H(a)} \ar{r}{g_1}
       \ar[dashrightarrow,bend left=40]{u}{\psi_a} 
    & (\P_1)_{\mathscr H(a)}, \ar[bend right=40,dashrightarrow]{u}[swap]{\psi_a}
  \end{tikzcd}
\]
where the section~$\psi_a$ is analytic  and defined over
the open subset $(U_1)_{\mathscr H(a)}$ of $(\P_1)^\an_{\mathscr H(a)}$.
Let $Y'_a$ be the irreducible component of~$Y_a$ that contains~$\psi_a(\xi_1)$;
it is geometrically irreducible.
Recall that $g_1$ has infinite order;
replacing~$g_1$ and~$g$ by some fixed power,
we may thus assume that $g\cdot Y'_a = Y'_a$. 

By proposition~\ref{prop.aut-curves}, either
$Y'_a\to(\P_1)_{\mathscr H(a)}$ is an isomorphism
and the components of its inverse are constant or homographies,
or 
there exists a subset~$\Delta$ of~$\mathbf P_1({\overline{\mathscr H(a)}})$
such that $\Card(\Delta)=2$ and $g_1(\Delta)=\Delta$
for every element $g=(g_1,\dots,g_n)\in G'_0\cap\Gamma$ such that $g\cdot W=W$
and $g\cdot Y'_a=Y'_a$.
Let us assume that we are in the latter case.
Using that $\Gamma_1\subset\PGL(2,F)$, we see
that $\Delta\subset\mathbf P_1(\overline F)$.
By lemma~\ref{lemm.few-rational}, 
the projection to~$\Gamma_1$ of the stabilizer of~$W$ inside~$G'_0\cap\Gamma$
has few rational points, contradicting
proposition~\ref{prop.stabilizer-infinite}.

We thus have shown that the components of the analytic map~$\psi_a$ 
are either constant or given by homographies.

Let $j\in\{m+1,\dots,n\}$.

Let us first assume that $q_j(\xi)\in\Omega_{\Gamma_j}$.
Then $g_j=\id$, whence the relation $\psi_{a,j}=\psi_{a,j} \circ g_1$.
Since $g_1\neq\id$, this implies that~$\psi_{a,j}$ is constant,
i.e., $\phi_j$ does not depend on the coordinate~$x_1$.
Since $U$ is reduced, the morphism~$\phi_j$ is deduced by pull-back of an 
analytic map $\theta_j\colon \prod_{i=2}^m U_i \to \P_1^\an$.

Let us then assume that $q_j(\xi)\in\mathscr L_{\Gamma_j}$.
Since the $j$th component of~$\phi$ takes the value~$q_j(\xi)$,
the section~$\psi_{a,j}$ cannot be constant. 
It is thus a homography~$\tau_{j,a}$.

A priori, one has  $\tau_{j,a}\in \PGL(2,\mathscr H(a))$ for every~$a$.
However, 
by condition~(3) of lemma~\ref{lemm.projection},
one has $\phi_j^{-1}(\mathscr L_{\Gamma_j})\subset\mathscr L_{\Gamma_1}$. 
The limit sets $\mathscr L_{\Gamma_1}$
and $\mathscr L_{\Gamma_j}$ are contained in $\P_1(F)$ and have
no isolated points, so that $\tau_{j,a}^{-1}$ maps an infinite subset
of $\P_1(F)$ into $\P_1(F)$; this implies that 
$\tau_{j,a}\in\PGL(2,F)$.

Observe that for~$x\in U_1\cap\P_1(F)$, one has
$\tau_{j,a}\cdot x=\psi_{a,j}(x)=\phi(x,a)$.
In particular, the assignment $a\mapsto\tau_{j,a}$  
is induced by an analytic morphism. Since it takes its values
in~$\PGL(2,F)$, it is constant.

Let $J'$ and $J''$ be the set of all $j\in\{m+1,\dots,n\}$ 
such that $q_j(\xi)$ belongs to $\mathscr L_{\Gamma_j}$
and to $\Omega_{\Gamma_j}$ respectively.
Let $\Omega'=\Omega_{\Gamma_1}\times\prod_{j\in J'}\Omega_{\Gamma_j}$
and $\Omega''=\prod_{i=2}^m\Omega_{\Gamma_i}\times\prod_{j\in J''}\Omega_{\Gamma_j}$; similarly, write
$X'=X_1\times\prod_{j\in J'}X_j$ 
and $X''=\prod_{i=2}^m X_i\times\prod_{j\in J''}X_j$,
and decompose the projection $p\colon\Omega\to X$
as $(p',p'')$, where $p'\colon\Omega'\to X'$ and $p''\colon\Omega''\to X''$
are the natural projections.

Let $Z'$ be the graph in $(\P_1\times \prod_{j\in J'}\P_1)^\an$
of $(\tau_j)_{j\in J'}$  
and $Z''\subset (\prod_{i=2}^m \P_1\times\prod_{j\in J''}\P_1)^\an$ 
be the graph of~$(\theta_j)_{j\in J''}$.
Let $Y'$ and $Y''$ be the Zariski closure of $Z'$ and $Z''$,
let $W'$ and $W''$ be their traces
in $\Omega'$ and $\Omega''$, 
and let $V'$ and $V''$ be the Zariski closures of $p'(Z')$
and $p''(Z'')$.
It is clear that $Y'=Z'$ is the curve  in $\P_1\times\prod_{j\in J'}\P_1$
(with coordinates $x_1$ and~$x_j$ for $j\in J'$)
given by the equations $x_j=\tau_j(x_1)$, and $W'$ is its trace on~$\Omega'$.
In particular, $W'$ is flat.

By construction, $Z'\times Z''$ is a subspace of~$Y^\an$
which meets $W$ in a Zariski dense subset of itself;
hence $Y=Y'\times Y''$ and $W=\Omega\cap Y^\an=W'\times W''$.
Moreover, $p(W)=p'(W')\times p''(W'')\subset V$, hence $V'\times V''\subset V$.
Consequently, $W''$
is a maximal algebraic irreducible subset of $(p'')^{-1}((V'')^\an)$.
By induction, $W''$ is flat.

Consequently, $W=W'\times W''$ is flat, as was to be shown.
\end{proof}

\section{A characterization of geodesic subvarieties}


\subsection{}
Let $F$ be a finite extension of $\Q_p$
and let $(\Gamma_i)_{1\leq i\leq n}$ be a finite family
of arithmetic Schottky subgroups of ranks $\geq 2$ in $\PGL(2,F)$
Let us set $\Omega=\prod_{i=1}^n\Omega_{\Gamma_i}$,
$X=\prod_{i=1}^n X_{\Gamma_i}$, and let $p\colon \Omega\to X^\an$
be the morphism deduced from the morphisms $p_{\Gamma_i}\colon\Omega_{\Gamma_i}\to X_{\Gamma_i}^\an$.

\begin{theo}\label{theo.geodesic}
Let $W$ be a Zariski closed subvariety of~$\Omega$, geometrically
irreducible.
Then the following properties are equivalent:
\begin{enumerate}\def\theenumi{\roman{enumi}}
\item The variety~$W$ is geodesic ;
\item Its projection $p(W)$ is algebraic;
\item The dimension of the Zariski closure of~$p(W)$ in~$X$
is equal to $\dim(W)$.
\end{enumerate}
\end{theo}
%
%
\begin{proof}
Let us assume that $W$ is geodesic and let us show that $p(W)$
is algebraic.

We may assume that no projection $p_{\Gamma_i}$
is constant on~$W$.
Define a relation~$\sim$ on $\{1,\dots,n\}$ given by $i\sim j$ if there
exists $g\in \PGL(2,F)$ (necessarily unique)
such that $g\Gamma_ig^{-1}$ and~$\Gamma_j$
are commensurable and $z_j=g\cdot z_i$  for every $z\in W$.
This is an equivalence relation.
Fix an element~$j$ in each equivalence class; for~$i$ such that
$i\sim j$, we may replace~$\Gamma_i$ by its conjugate $g \Gamma_i g^{-1}$
and assume that $z_j=z_i$ on~$W$.
This shows that $W$ and~$\Omega$ decompose as a product indexed by the set
of equivalence classes of the following   particular situation:
all the subgroups $\Gamma_i$ are commensurable, and
$W$ is the diagonal of~$\Omega$.  
It thus suffices to treat this particular case.

Let $\Gamma_0=\bigcap_i \Gamma_i$
and $X_0$ be the algebraic curve associated with $\Omega_{\Gamma_0}/\Gamma_0$.
Then, for every~$i$, the morphism $f_i\colon W\to X_i^\an$
deduced from~$f=p|_W$
factors as the composition of the 
uniformization $p_0\colon \Omega_{\Gamma_0}\to X_0^\an$
and of a  finite morphism $X_0^\an \to X_i^\an$.
By \textsc{gaga} (\cite{berkovich1990}, corollary 3.5.2;
\cite{poineau2010a}, appendix),
a finite analytic morphism of algebraic curves 
is algebraic;  consequently, there exists a finite morphism
$q_i\colon X_0\to X_i$ such that $f_i=q_i^\an\circ p_0$.
Then $p(W)$  is the image of $X_0$ by the
finite morphism $q=(q_1,\dots,q_n)\colon X_0\to X$, hence is algebraic.
This shows that (i) implies (ii). Since it is clear that (ii)  implies (iii), it remains to prove that (iii) implies (i).
%

Let us assume now that the dimension of the Zariski closure~$V$ of~$p(W)$ 
in~$X$ is equal to the dimension of~$W$.
By construction, $W$ is a maximal irreducible algebraic
subvariety of~$p^{-1}(V^\an)$. By proposition~\ref{mt-bis}, $W$ is flat.
A similar analysis as in the proof of the first implication
shows that there is a partition of the indices $\{1,\dots,n\}$
under which $W$ decomposes as a product of flat curves and points.
Since it suffices to prove that each of these curves is geodesic,
we may assume that $W$ is a flat curve,
of the form
\[ W=\{(z,g_2\cdot z,\dots, g_n\cdot z)\}\cap\Omega, \]
where $g_2,\dots,g_n\in\PGL(2,F)$.

Let us first assume that $n=2$.
Let then $g\in\PGL(2,F)$ be
such that $W=\{(z,g\cdot z)\}\cap\Omega$ and let us prove that
$\Gamma_2$ and $g\Gamma_1g^{-1}$ are commensurable,
a property which is equivalent to the finiteness of both orbit sets
$\Gamma_2\backslash \Gamma_2 g\Gamma_1$ and $\Gamma_1\backslash \Gamma_1 g^{-1}\Gamma_2$.

Let us argue by contradiction and assume that $\Gamma_2\backslash \Gamma_2 g\Gamma_1$ is infinite. (The other finiteness is analogous, or follows by symmetry.)
Fix a rigid point $z\in\Omega_{\Gamma_1}$.
Let $A\subset\Gamma_1$ be a set such that $gA$ is a set of representatives
of $\Gamma_2\backslash \Gamma_2 g \Gamma_1$; by assumption, $A$ is infinite.
Since $\Gamma\backslash W\subset V^\an$, the algebraic variety~$V$ contains
the infinite set of points
$ p(a\cdot z,g\cdot az) = (p_1(z),p_2(ga\cdot z))$,
for $a\in A$, hence it contains its Zariski closure $\{p_1(z)\}\times X_2$.
Since this holds for every $z\in W$, we deduce that $V$ contains $X_1\times X_2$,
contradicting the assumption that $\dim(W)=1$.

Let us now return to the general case.
To prove that $W$ is geodesic, it suffices to establish that
the subgroups $\Gamma_j$ and $g_j\Gamma_1 g_j^{-1}$ are commensurable,
for every $j\in\{2,\dots,n\}$.
Up to renumbering the indices, it suffices to treat the case $j=2$.
Let $\Omega'=\Omega_{\Gamma_1}\times\Omega_{\Gamma_2}$,
let $p'\colon \Omega'\to X'=X_1\times X_2$ be the uniformization
map,  and denote by~$\pi$ the projections
from~$\Omega$ to~$\Omega'$ and from $X$ to~$X'$.
Let $W'=\pi(W)$ and $V'=\pi(V)$. 
By Chevalley's theorem, $V'$ is an algebraic curve in~$X'$.
Obviously, $W'$ is a flat curve contained in $(p')^{-1}((V')^\an)$,
hence is an maximal irreducible  algebraic subset
of $(p')^{-1}((V')^\an)\cap\Omega'$. 
By the case $n=2$, the Schottky groups $\Gamma_2$ and $g_2 \Gamma_1 g_2^{-1}$ 
are commensurable, as was to be shown.
This concludes the proof of theorem~\ref{theo.geodesic}.
\end{proof}

\begin{coro}
Let $V$ be an irreducible \emph{curve} in $X$.
Then every  irreducible algebraic subvariety of~$\Omega_{\C_p}$
which is maximal among those contained in~$p^{-1}(V^\an_{\C_p})$ is geodesic.
\end{coro}
\begin{proof}
Let $W_0$ be a irreducible algebraic subvariety of~$\Omega_{\C_p}$,
maximal among those contained in $p^{-1}(V^\an_{\C_p})$;  
let us prove that $W_0$
is geodesic. We may assume that $\dim(W_0)>0$.
Since $p$ is surjective and has discrete fibers, one has $\dim(p^{-1}(V^\an_{\C_p}))=\dim(V^\an_{\C_p})$, hence $\dim(W_0)=1$, so that $W_0$
is an irreducible component of $p^{-1}(V^\an)_{\C_p}$.
By theorem~7.16 of~\cite{ducros2009}, there exists a finite extension~$E$
of~$F$ and an irreducible component~$W$ of~$p^{-1}(V^\an)_E$ such that
$W_0=W_{\C_p}$. 

By theorem~\ref{theo.geodesic}, $W$ is geodesic. Consequently,
$W_0$ is geodesic.
\end{proof}

\begin{rema}
This corollary suggests that the main results of the paper
extend to maximal algebraic irreducible subvarieties of $p^{-1}(V^\an)_{\C_p}$,
without assuming that they are defined over a finite extension of~$F$.
\end{rema}

\bibliographystyle{smfplain}
\bibliography{aclab,acl,drinfeld}
\end{document}